\documentclass{amsart}
\usepackage{amssymb,amstext,amsmath,amscd,amsthm,amsfonts,enumerate,graphicx,latexsym,stmaryrd}
\usepackage[usenames]{color}
\usepackage[all]{xy}
\newtheorem{thm}{Theorem}[section]
\newtheorem{lem}[thm]{Lemma}
\newtheorem{prop}[thm]{Proposition}
\newtheorem{cor}[thm]{Corollary}
\theoremstyle{definition}
\newtheorem{dfn}[thm]{Definition}

\newtheorem{setup}[thm]{Setup}
\newtheorem{rem}[thm]{Remark}

\newtheorem{ex}[thm]{Example}
\theoremstyle{remark}
\newtheorem*{ac}{Acknowlegments}
\newtheorem*{conv}{Convention}

\numberwithin{equation}{thm}
\def\A{\mathrm{A}}
\def\ann{\operatorname{Ann}}
\def\ass{\operatorname{Ass}}
\def\b{\operatorname{B}}

\def\c{\mathsf{C}}
\def\ccm{\operatorname{C_{CM}}}
\def\cd{\operatorname{cd}}
\def\ch{\operatorname{char}}
\def\Cl{\operatorname{Cl}}
\def\cl{\operatorname{cl}}
\def\cm{\operatorname{MCM}}

\def\codim{\operatorname{codim}}
\def\cx{\operatorname{cx}}
\def\D{\mathrm{D}}
\def\depth{\operatorname{depth}}
\def\e{\operatorname{e}}
\def\edim{\operatorname{edim}}
\def\Ext{\operatorname{Ext}}
\def\g{\operatorname{G_0}}
\def\ge{\geqslant}
\def\grade{\operatorname{grade}}
\def\H{\operatorname{H}}
\def\h{\mathsf{H}}
\def\height{\operatorname{ht}}
\def\Hom{\operatorname{Hom}}
\def\I{\operatorname{I}}
\def\im{\operatorname{Im}}
\def\int{\operatorname{Int}}
\def\jac{\operatorname{jac}}
\def\k{\mathsf{K}}
\def\ker{\operatorname{Ker}}
\def\le{\leqslant}
\def\lhom{\operatorname{\underline{Hom}}}
\def\m{\mathfrak{m}}
\def\mod{\operatorname{mod}}

\def\NN{\mathbb{N}}
\def\nn{\mathfrak{n}}
\def\p{\mathfrak{p}}
\def\pd{\operatorname{pd}}
\def\Proj{\operatorname{Proj}}
\def\Q{\mathbb{Q}}
\def\q{\mathfrak{q}}
\def\R{\mathbb{R}}
\def\rk{\operatorname{rk}}
\def\rmm{\mathbb{R}_{\le0}}
\def\RP{\mathbb{R}_{>0}}
\def\rp{\mathbb{R}_{\ge0}}

\def\sing{\operatorname{Sing}}
\def\spec{\operatorname{Spec}}
\def\supp{\operatorname{Supp}}
\def\syz{\mathrm{\Omega}}
\def\Tor{\operatorname{Tor}}
\def\v{\mathrm{V}}
\def\vv{\upsilon}
\def\Z{\mathbb{Z}}
\def\ZP{\mathbb{Z}_{>0}}
\def\zp{\mathbb{Z}_{\ge0}}
\begin{document}
\allowdisplaybreaks
\title{Grothendieck groups, convex cones and maximal Cohen--Macaulay points}
\author{Ryo Takahashi}
\address{Graduate School of Mathematics, Nagoya University, Furocho, Chikusaku, Nagoya 464-8602, Japan}
\email{takahashi@math.nagoya-u.ac.jp}
\urladdr{https://www.math.nagoya-u.ac.jp/~takahashi/}
\thanks{The author was partly supported by JSPS Grant-in-Aid for Scientific Research 19K03443 and JSPS Fund for the Promotion of Joint International Research 16KK0099}
\dedicatory{Dedicated to Professor Yuji Yoshino on the occasion of his retirement}
\begin{abstract}
Let $R$ be a commutative noetherian ring.
Let $\h(R)$ be the quotient of the Grothendieck group of finitely generated $R$-modules by the subgroup generated by pseudo-zero modules.
Suppose that the $\R$-vector space $\h(R)_\R=\h(R)\otimes_\Z\R$ has finite dimension.
Let $\c(R)$ (resp. $\c_r(R)$) be the convex cone in $\h(R)_\R$ spanned by maximal Cohen--Macaulay $R$-modules (resp. maximal Cohen--Macaulay $R$-modules of rank $r$).
We explore the interior, closure and boundary, and convex polyhedral subcones of $\c(R)$.
We provide various equivalent conditions for $R$ to have only finitely many rank $r$ maximal Cohen--Macaulay points in $\c_r(R)$ in terms of topological properties of $\c_r(R)$.
Finally, we consider maximal Cohen--Macaulay modules of rank one as elements of the divisor class group $\Cl(R)$.
\end{abstract}
\maketitle
\tableofcontents
\section{Introduction}

Let $R$ be a commutative noetherian ring.
Let $\g(R)$ be the Grothendieck group of finitely generated $R$-modules.
Kurano \cite{K2} defines the {\em Grothendieck group modulo numerical equivalence} $\overline{\g(R)}$ and develops its theory.
Chan and Kurano \cite{CK} introduce and explore the {\em Cohen--Macaulay cone} $\ccm(R)$, which is by definition the convex cone in $\overline{\g(R)}_\R=\overline{\g(R)}\otimes_\Z\R$ spanned by maximal Cohen--Macaulay $R$-modules.

Let $\h(R)$ be the quotient of $\g(R)$ by the subgroup generated by pseudo-zero modules.
Let $\c(R)$ be the convex cone in $\h(R)_\R$ spanned by maximal Cohen--Macaulay modules.
When $R$ is a Cohen--Macaulay local ring with $\dim R\le3$, the canonical map $\g(R)\to\overline{\g(R)}$ factors through $\h(R)$.
Hence the studies of $\h(R),\c(R)$ should contribute to those of $\overline{\g(R)},\ccm(R)$.
We thus explore $\h(R),\c(R)$ in this paper.

Suppose that $R$ is an integral domain.
Then taking the rank of each $R$-module induces the rank function $\rk:\h(R)\to\Z$, and it holds that $\h(R)=\Z[R]+\k(R)\cong\Z\oplus\k(R)$, where $\k(R)=\ker(\rk)$; note that $\k(R)$ is isomorphic to the divisor class group $\Cl(R)$ in the case where $R$ is normal.
Suppose further that $\dim_\R\h(R)_\R<\infty$, or equivalently, that $\dim_\R\k(R)_\R<\infty$.
Then $\h(R)_\R=\R[R]+\k(R)_\R$ can be regarded as an Euclidean space, where $[R]$ and a basis of $\k(R)_\R$ form a normal orthogonal basis.

First of all, we investigate the interior, closure and boundary of $\c(R)$.

\begin{thm}[Propositions \ref{48}, \ref{35} and \ref{44}]\label{b}
Let $R$ be a Cohen--Macaulay normal local ring with a canonical module $\omega$.
Suppose that $\h(R)_\R$ is a finite-dimensional $\R$-vector space.
Then the following hold.
\begin{enumerate}[\rm(1)]
\item
The points $[R]$ and $[\omega]$ belong to the interior of $\c(R)$.
\item
The closure of $\c(R)$ is a strongly convex cone if and only if it meets $\k(R)_\R$ only at the origin.
\item
Suppose that $R$ is Gorenstein.
The convex cone $\c(R)$ and its closure are symmetric with respect to the axis $\R[R]$.
Let $M$ be a maximal Cohen--Macaulay $R$-module.
If $[M]$ or $[M^\ast]$ belongs to the boundary of $\c(R)$, then the ranks of the syzygies and cosyzygies of $M$ are more than or equal to the rank of $M$.
\end{enumerate}
\end{thm}

It is unknown if $\ccm(R)$ is polyhedral.
If $\c(R)$ is polyhedral and $R$ is a Cohen--Macaulay local ring with $\dim R\le3$, then $\ccm(R)$ is polyhedral as well.
Thus it is interesting to think of the question asking if $\c(R)$ is polyhedral.
We do not know if this question is itself affirmative, but show the following result.

\begin{thm}[Theorem \ref{11}]\label{d}
Let $R\subseteq S$ be a finite extension of Cohen--Macaulay local domains.
Suppose that $R$ is normal, $\Hom_R(S,R)$ is maximal Cohen--Macaulay, $\h(S)_\R$ is finite-dimensional, and $\c(S)$ is polyhedral.
Let $V$ be the convex polyhedral subcone of $\c(R)$ spanned by the images of the generators of $\c(S)$.
Let $M$ be an $R$-module which has unbounded Betti numbers and is locally free in codimension one.
Then $[\syz_R^nM]$ is in $V$ for infinitely many $n\in\NN$, where $\syz_R^nM$ denotes the $n$th syzygy of the $R$-module $M$.
\end{thm}

\noindent
As a special case of this theorem, it holds that for a Gorenstein normal local ring $R$ possessing a simple singularity as a finite extension, $\c(R)$ admits a convex polyhedral subcone $V$ such that for all $R$-modules $M$ having unbounded Betti numbers there exist infinitely many integers $n\ge0$ with $[\syz_R^nM]\in V$.

Dao and Kurano \cite{DK} (or possibly some other people) conjecture that under some mild assumptions a Cohen--Macaulay normal local ring $R$ has only finitely many maximal Cohen--Macaulay modules of rank one up to isomorphism.
They investigate the Cohen--Macaulay cone $\ccm(R)$ to deduce that the conjecture holds true for a hypersurface of dimension three with an isolated singularity.

Let $\c_r(R)$ be the convex cone in $\h(R)_\R$ spanned by maximal Cohen--Macaulay modules of rank $r$; note that $\c(R)=\sum_{r>0}\c_r(R)$.
For a positive integer $s$, a {\em rank $s$ maximal Cohen--Macaulay point} means a point represented by a rank $s$ maximal Cohen--Macaulay module.
We explore the convex cone $\c_r(R)$ to relate it with the conjecture of Dao and Kurano mentioned above.
We prove the following theorem.

\begin{thm}[Theorem \ref{1}]\label{a}
Let $R$ be a noetherian domain such that $\h(R)$ is finitely generated.
Let $r$ be a positive integer.
Then the following five conditions are equivalent.
\begin{enumerate}[\rm(1)]
\item
There are only finitely many rank $r$ maximal Cohen--Macaulay points in $\c_r(R)$.
\item
The convex cone $\c_r(R)$ is polyhedral.
\item
The convex cone $\c_r(R)$ is a closed subset of $\h(R)_\R$.
\item
The closure of $\c_r(R)$ meets $\k(R)_\R$ only at the origin.
\item
The set of points in $\c_r(R)$ with rank $r$ is bounded.
\end{enumerate}
When $R$ is a normal Cohen--Macaulay ring with a canonical module, {\rm(4)} is equivalent to saying that the closure of $\c_r(R)$ is a strongly convex cone.
When $r=1$ and $R$ is normal, {\rm(1)} is equivalent to saying that there exist only finitely many isomorphism classes of maximal Cohen--Macaulay $R$-modules of rank one.
\end{thm}

\noindent
As an application of this theorem, we obtain the following:
Let $R$ be a homomorphic image of a Gorenstein local ring.
Assume $R$ is a normal ring of dimension at most three such that $\Cl(R)$ is a finitely generated abelian group of rank one.
Suppose there exist an $R$-module $L$ of finite length and finite projective dimension and a torsion $R$-module $T$ such that the intersection multiplicity $\chi(L,T)$ is nonzero.
Then there exist only a finite number of nonisomorphic maximal Cohen--Macaulay $R$-modules of rank one.

Next, we try to approach the Dao--Kurano conjecture in a more direct way.
We denote by $\cd I$ the {\em cohomological dimension} of $I$, that is to say, the supremum of integers $i$ such that the local cohomology $\H_I^i(R)$ does not vanish.
We consider maximal Cohen--Macaulay points on the line in $\Cl(R)$ defined by a reflexive ideal that satisfies a certain condition on cohomological dimension or asymptotic depth.

\begin{thm}[Theorems \ref{8} and \ref{3}]\label{c}
Let $R$ be a $d$-dimensional Gorenstein normal local ring.
Let $I$ be a nonzero reflexive ideal of $R$.
\begin{enumerate}[\rm(1)]
\item
Let $R$ have an isolated singularity and $d\ge2$.
Assume either that the limit of $\depth R/I^n$ is nonzero, or that $R$ is analytically irreducible and $\cd I\ne1$.
Then there exist only a finite number of maximal Cohen--Macaulay points on $\Z I$ in $\Cl(R)$.
In the latter case, one also has an isomorphism $\Z I\cong\Z$.
\item
Let $d\ge3$.
Suppose that $I$ is non-principal, Gorenstein and locally free on the punctured spectrum of $R$.
Then the point $nI\in\Cl(R)$ with $n\in\Z$ is maximal Cohen--Macaulay if and only if $n=-1,0,1$.
\end{enumerate}
\end{thm}

\noindent
This theorem especially says that if such an ideal $I$ as in (1) or (2) satisfies the equality $\Cl(R)=\Z I$, then there exist only finitely many maximal Cohen--Macaulay $R$-modules of rank one up to isomorphism.

The organization of this paper is as follows.
Let $R$ be a commutative noetherian ring.
In Section 2, we state several basic properties of the abelian group $\h(R)$ and our ambient space $\h(R)_\R$.
In Section 3, we investigate topological properties of our convex cone $\c(R)$, and prove Theorem \ref{b}.
In Section 4, we consider when a given module belongs to a fixed convex polyhedral subcone of $\c(R)$, and prove Theorem \ref{d}.
In Section 5, we study topological properties of the convex cone $\c_r(R)$ in a similar context to Section 3, and prove Theorem \ref{a}.
In Section 6, we explore maximal Cohen--Macaulay points in $\Cl(R)$, and prove Theorem \ref{c}.
In this paper, we also construct various concrete examples to which our results apply.

\begin{conv}
Throughout the paper, we adopt the following convention.

We assume that all rings are commutative and noetherian, and that all modules are finitely generated.
Let $R$ be a (commutative noetherian) ring.
Let $\mod R$ be the category of (finitely generated) $R$-modules, and $\cm(R)$ the full subcategory consisting of maximal Cohen--Macaulay $R$-modules; recall that an $R$-module $M$ is called {\em maximal Cohen--Macaulay} if $\depth M_\p\ge\height\p$ for all $\p\in\spec R$ (note that $\depth0=\infty$).
Denote by $\g(R)$ the Grothendieck group of $\mod R$.
We set $(-)^\ast=\Hom_R(-,R)$.
We say that $R$ has {\em finite (resp. countable) Cohen--Macaulay representation type} if it has only finitely (resp. countably) many nonisomorphic indecomposable maximal Cohen--Macaulay modules.
For a $\Z$-module $G$, we set $G_\Q=G\otimes_\Z\Q$ and $G_\R=G\otimes_\Z\R$.
We often omit subscripts/superscripts if there is no danger of confusion.

Let $R$ be a local ring.
The {\em codimensions} of the ring $R$, an ideal $I$ of $R$ and an $R$-module $M$ are defined by $\codim R=\edim R-\dim R$, $\codim I=\dim R-\dim R/I$ and $\codim M=\dim R-\dim M$.
For an integer $n\ge0$, we denote by $\syz^n_RM$ and $\beta_n^R(M)$ the $n$th syzygy and the $n$th Betti number of $M$ in the minimal free resolution of $M$, respectively.
Note that the equality $\beta_0^R(\syz^n_RM)=\beta_n^R(M)$ holds.

Let $R$ be a domain (resp. a normal domain).
Let $M\ne0$ be a maximal Cohen--Macaulay $R$-module.
Then, $M$ has rank one if and only if $M$ is isomorphic to an ideal (resp. a reflexive ideal) of $R$ (see \cite[Proposition 1.4.1]{BH}).
Thus, in this paper, we often identify a maximal Cohen--Macaulay module of rank one and a nonzero ideal (resp. a nonzero reflexive ideal) that is a maximal Cohen--Macaulay module.

Whenever a Cohen--Macaulay ring $R$ admits a canonical module $\omega$, we set $(-)^\dag=\Hom_R(-,\omega)$.
\end{conv}

\section{The metric vector space $\h(R)_\R$}

In this section, we introduce the vector space $\h(R)_\R$ and investigate its basic properties.
To explain the motivation to study $\h(R)_\R$, we recall the definitions of the Grothendieck group modulo numerical equivalence and the intersection multiplicity of modules.

\begin{dfn}
Let $R$ be a local ring.
\begin{enumerate}[(1)]
\item
An $R$-module $M$ is called {\em numerically trivial} if $\chi_P(M)=\sum_{i=0}^\infty(-1)^i\ell_R(\H_i(P\otimes_RM))$ vanishes for all perfect complexes $P$ with finite length homologies.
For each such $P$, the assignment $M\mapsto\chi_P(M)$ extends to a map $\g(R)\to\Z$, and the notion of numerical triviality is extended to elements of $\g(R)$.
\item
The {\em Grothendieck group modulo numerical equivalence} in the sense of \cite{K2} is by definition the quotient
$$
\overline{\g(R)}=\g(R)/{\langle x\in\g(R)\mid\text{$x$ is numerically trivial}\rangle}
$$
of the Grothendieck group $\g(R)$.
\item
Let $M,N$ be $R$-modules.
Suppose that $M$ has finite projective dimension and $M\otimes_RN$ has finite length.
Then the alternating sum
$$\textstyle
\chi(M,N)=\sum_{i=0}^\infty(-1)^i\ell_R(\Tor_i^R(M,N))
$$
is well-defined.
This is called the {\em intersection multiplicity} of $M$ and $N$.
The assignment $M\mapsto\chi(M,N)$ induces a homomorphism $\chi(M,-):\g(R)\to\Z$ of $\Z$-modules.
\end{enumerate}
\end{dfn}

Now we introduce the $\Z$-module $\h(R)$ and the rank function on it.

\begin{dfn}
\begin{enumerate}[(1)]
\item
A {\em pseudo-zero} $R$-module in the sense of Bourbaki \cite[VII \S4]{Bou} is by definition an $R$-module $M$ such that $M_\p=0$ for all $\p\in\spec R$ with $\height\p\le1$.
This condition is equivalent to saying that $\height\p\ge2$ for all $\p\in\supp_RM$, which is also equivalent to simply saying that $\height(\ann M)\ge2$, where $\height R:=\infty$.
We define the group $\h(R)$ as the quotient of the Grothendieck group $\g(R)$ by the subgroup generated by the classes of pseudo-zero $R$-modules.
\item
Let $R$ be a domain.
Then taking the rank of each $R$-module defines the rank function $\rk:\h(R)\to\Z$.
We denote by $\k(R)$ the kernel of this $\Z$-module homomorphism.
It then holds that $\h(R)=\Z[R]+\k(R)\cong\Z\oplus\k(R)$.
The rank function is extended via $(-)_\R$ to the function $\rk_\R:\h(R)_\R\to\R$.
\end{enumerate}
\end{dfn}

We establish a lemma on the intersection multiplicity of modules.

\begin{lem}\label{30-0}
\begin{enumerate}[\rm(1)]
\item
Let $R$ be a Cohen--Macaulay local ring.
Then an $R$-module $M$ is numerically trivial if and only if $\chi(L,M)=0$ for every $R$-module $L$ of finite length and finite projective dimension.
\item
Let $R$ be a local ring of dimension at most three.
Let $L$ be an $R$-module of finite length and finite projective dimension.
Then the $\Z$-homomorphism $\chi(L,-):\g(R)\to\Z$ induces a $\Z$-homomorphism $\chi(L,-):\h(R)\to\Z$, which extends to an $\R$-linear map $\chi(L,-):\h(R)_\R\to\R$ by the functor $(-)_\R$.
\end{enumerate}
\end{lem}

\begin{proof}
(1) The assertion is shown in \cite[Proposition 2]{RS}.

(2) Let $M$ be a pseudo-zero $R$-module.
If $\dim R\le1$, then $M=0$, whence $\chi(L,M)=0$.
Let $\dim R$ be $2$ or $3$.
Then we see that $\dim M\le1$, and $\dim L+\dim M<\dim R$.
It follows from \cite[Theorem (1.1) and the preceding part]{Fx} that $\chi(L,M)=0$.
Thus the $\Z$-homomorphism $\chi(L,-):\h(R)\to\Z$ is induced.
\end{proof}

Applying this lemma, we immediately obtain the following proposition.

\begin{prop}\label{47}
Let $R$ be a Cohen--Macaulay local ring of dimension at most three.
Then the canonical surjection $\g(R)\to\overline{\g(R)}$ factors through $\h(R)$.
\end{prop}

This proposition explains our motivation to study $\h(R)$; various properties of $\h(R)$ are conveyed to $\overline{\g(R)}$ through the surjective homomorphism $\h(R)\to\overline{\g(R)}$ in the case where $R$ is a Cohen--Macaulay local ring of dimension at most three.

Next we consider what ring homomorphisms $R\to S$ induce $\Z$-module homomorphisms $\h(S)\to\h(R)$.

\begin{prop}\label{39}
\begin{enumerate}[\rm(1)]
\item
Let $I$ be an ideal of $R$.
Then the canonical map $\g(R/I)\to\g(R)$ given by $[M]\mapsto[M]$ induces a homomorphism $\h(R/I)\to\h(R)$ of $\Z$-modules.
\item
Let $f:(R,\m)\to(S,\nn)$ be a finite local homomorphism of local rings.
The canonical map $\g(S)\to\g(R)$ given by $[X]\mapsto[X]$ induces a homomorphism $g:\h(S)\to\h(R)$ of $\Z$-modules.
If $f$ is injective, then $g_\Q:\h(S)_\Q\to\h(R)_\Q$ is surjective and so is $g_\R:\h(S)_\R\to\h(R)_\R$.
\end{enumerate}
\end{prop}

\begin{proof}
(1) Let $M$ be an $R/I$-module satisfying the inequality $\height(\ann_{R/I}M)\ge2$.
Then $\ann_RM$ contains $I$, and we have $(\ann_RM)/I=\ann_{R/I}M$.
It is seen that $\height(\ann_RM)\ge2$.

(2) By (1) we may assume that $f$ is injective.
Let $X$ be an $S$-module, and take $\p\in\supp_RX$.
We find $\q\in\spec S$ with $\q\cap R=\p$ by \cite[Theorem 9.3(i)]{M}, and $\height\p\ge\height\q$ by \cite[Exercise 9.8]{M}.
As $X_\p=(X_\q)_\p$, we have $\q\in\supp_SX$.
Thus, if $X$ is pseudo-zero as an $S$-module, then it is pseudo-zero as an $R$-module.
We obtain a $\Z$-homomorphism $g:\h(S)\to\h(R)$ given by $g([X])=[X]$.

Let us show that $g_\Q$ is surjective.
Fix an $R$-module $M$.
Taking a Bourbaki filtration of $M$, we can write $[M]=\sum_{i=1}^n[R/\p_i]$ in $\h(R)$, where $\p_i\in\supp_RM$ and $\height\p_i\le1$.
We find $\q_i\in\spec S$ with $\q_i\cap R=\p_i$.
The map $f$ induces a finite map $R/\p_i\hookrightarrow S/\q_i$.
There is an exact sequence $0\to(R/\p_i)^{\oplus r_i}\to S/\q_i\to C_i\to0$ of $R/\p_i$-modules with $r_i=\rk_{R/\p_i}S/\q_i$.
Note that $\supp_RC_i\subseteq\v(\p_i)\setminus\{\p_i\}$.
We proceed with two steps.

(i) Suppose that $\height\p\ge1$ for all $\p\in\supp_RM$.
Then for each $1\le i\le n$ we have $\height\p_i=1$, and see that $C_i$ is pseudo-zero as an $R$-module.
We get $[S/\q_i]=r_i[R/\p_i]$ in $\h(R)$, and $[R/\p_i]=g_\Q(r_i^{-1}[S/\q_i])$.
It follows that $[M]=\sum_{i=1}^n[R/\p_i]$ belongs to the image of the map $g_\Q$.

(ii) Now we consider the general case.
Note that $\height\p\ge1$ for all $\p\in\supp_RC_i$ and all $1\le i\le n$.
By (i) each $[C_i]$ belongs to $\im g_\Q$.
Hence $[R/\p_i]=r_i^{-1}([S/\q_i]-[C_i])\in\im g_\Q$, and $[M]=\sum_{i=1}^n[R/\p_i]\in\im g_\Q$.
\end{proof}

Here we recall the definition of the determinant map.

\begin{dfn}
Let $R$ be a normal domain.
The {\em determinant map} (or the {\em first Chern class}) $\cl:\h(R)\to\Cl(R)$ is defined, and the maps $\cl:\k(R)\to\Cl(R)$ and $\binom{\rk}{\cl}:\h(R)\to\Z\oplus\Cl(R)$ are isomorphisms; see \cite[Proposition 2.11]{ncr}, \cite[Remark 2.1]{K} and \cite[Lemma (13.3)]{Y}.
For a torsion $R$-module $T$ one has
$$\textstyle
\cl(T):=\cl([T])=-\sum_{\p\in\spec R,\,\height\p=1}\ell_{R_\p}(T_\p)\cdot\p,
$$
for which we refer to \cite[VII \S4]{Bou} and \cite[the proof of Proposition 3.1(2)]{qgor}.
Note that $\cl(R/I)=-I$ for a nonzero reflexive ideal $I$ of $R$.
The determinant map is extended to a map $\cl_\R:\h(R)_\R\to\Cl(R)_\R$.
\end{dfn}

We establish a lemma which is used not only in the next proposition but also later.

\begin{lem}\label{61}
Let $R$ be a normal domain.
Let $M$ be a torsion $R$-module.
Let $N$ be an $R$-module such that $N_\p$ is $R_\p$-free for all $\p\in\spec R$ with $\height\p=1$.
Then $[\Ext_R^1(M,N)]=(\rk_RN)[M]$ in $\h(R)$.
\end{lem}

\begin{proof}
Fix $\p\in\spec R$ with $\height\p=1$.
Then $R_\p$ is a discrete valuation ring.
Put $r=\rk_R N$, and notice $N_\p\cong R_\p^{\oplus r}$.
There are isomorphisms $\Ext_R^1(M,N)_\p\cong\Ext_{R_\p}^1(M_\p,N_\p)\cong\Ext_{R_\p}^1(M_\p,R_\p)^{\oplus r}\cong M_\p^{\oplus r}$; the last isomorphism follows from the structure theorem of finitely generated modules over a principal ideal domain.
Hence $\cl(\Ext_R^1(M,N))=\cl(M^{\oplus r})$, which implies $[\Ext_R^1(M,N)]=[M^{\oplus r}]=r[M]$ in $\h(R)$.
\end{proof}

We introduce an endomorphism of $\h(R)$ and its extension to $\h(R)_\R$.

\begin{dfn}
Let $R$ be a normal domain and $N$ an $R$-module.
For an $R$-module $M$ and an integer $i\ge0$, the $R$-module $\Ext_R^i(M,N)$ is pseudo-zero if either so is $M$ or $i\ge2$ (as $R_\p$ is regular for each $\p\in\spec R$ with $\height\p\le1$).
We can define an endomorphism $\vv_N:\h(R)\to\h(R)$ of the $\Z$-module $\h(R)$ by
$$
\vv_N([M])=[\Hom_R(M,N)]-[\Ext_R^1(M,N)].
$$
Extending this, we get a linear transformation $(\vv_N)_\R:\h(R)_\R\to\h(R)_\R$ of the $\R$-vector space $\h(R)_\R$.
If $N$ has rank $1$ and $N_\p$ is $R_\p$-free for all $\p\in\spec R$ with $\height\p=1$, then $\vv_N(t)=-t$ for each $t\in\k(R)$ by Lemma \ref{61}, and in particular, $\vv_N(\k(R))=\k(R)$ and $(\vv_N)_\R(\k(R)_\R)=\k(R)_\R$.
\end{dfn}

The following proposition describes relationships in $\h(R)$ of modules and their duals.

\begin{prop}\label{21}
Let $R$ be a normal domain.
Let $M$ be an $R$-module of rank $r$ such that $M_\p$ is $R_\p$-free for all $\p\in\spec R$ with $\height\p=1$.
Then there is an equality $[M]+[M^\ast]=2r[R]$ in $\h(R)$.
If $R$ is a Cohen--Macaulay ring with a canonical module $\omega$, then one has $[M]+[M^\dag]=r([R]+[\omega])$ in $\h(R)$.
\end{prop}

\begin{proof}
An exact sequence $0\to R^{\oplus r}\to M\to C\to0$ with $C$ torsion gives $[M]=r[R]+[C]$ in $\g(R)$.
We have $[M^\ast]=\vv_R([M])=r\cdot\vv_R([R])+\vv_R([C])=r[R]-[C]$ in $\h(R)$, which shows the first assertion.
The second assertion follows from the equalities $[M^\dag]=\vv_\omega([M])=r\cdot\vv_\omega([R])+\vv_\omega([C])=r[\omega]-[C]$.
\end{proof}

From now on, we consider the extension $\h(R)_\R$ of $\h(R)$, which is an $\R$-vector space.

\begin{setup}\label{20}
Let $R$ be a domain.
\begin{enumerate}[(1)]
\item
Suppose that $\h(R)$ is finitely generated, or equivalently, that $\k(R)$ is so.
Then there exist torsion $R$-modules $\Phi_1,\dots,\Phi_\zeta,\Psi_1,\dots,\Psi_\eta$ such that
$$\textstyle
\h(R)
=\Z[R]+\k(R)
=\Z[R]+\sum_{i=1}^\zeta\Z[\Phi_i]+\sum_{j=1}^\eta\Z[\Psi_j]
\cong\Z^{\oplus(\zeta+1)}\oplus\bigoplus_{j=1}^\eta\Z/\theta_j\Z,
$$
where $\theta_j$ stands for the order of $[\Psi_j]$ in the $\Z$-module $\h(R)$.
\item
Assume $\dim_\R\h(R)_\R<\infty$, i.e., $\dim_\R\k(R)_\R<\infty$.
Let $[\Phi_1],\dots,[\Phi_\zeta]$ be an $\R$-basis of $\k(R)_\R$.
Then
$$
\h(R)_\R=\R[R]+\R[\Phi_1]+\cdots+\R[\Phi_\zeta]\cong\R^{\oplus(\zeta+1)}.
$$
We define a metric on $\h(R)_\R$ by using the metric of the Euclidean space $\R^{\oplus(\zeta+1)}$.
Thus $\h(R)_\R$ is a metric vector space with $[R],[\Phi_1],\dots,[\Phi_\zeta]$ the normal orthogonal basis.
\end{enumerate}
Whenever $\h(R)_\R$ is finite-dimensional, we adopt the notation of (2).
Note that the $\Phi_i$ in (1) satisfy the condition on the $\Phi_i$ in (2).
Whenever $\h(R)$ is finitely generated, we use the $\Phi_i$ taken as in (1).
\end{setup}

Clearly, $\h(R)_\R$ is finite-dimensional if $R$ is normal and $\Cl(R)_\R$ is finite-dimensional.
The example below says that even if $R$ is non-normal, $\h(R)_\R$ can be finite-dimensional.

\begin{ex}
Let $k$ be a field.
\begin{enumerate}[\rm(1)]
\item
Assume $\ch k\ne2$.
Let $R=k[\![x^4,x^3y,xy^3,y^4]\!]$ be a pinched Veronese subring of the formal power series ring $S=k[\![x,y]\!]$.
The integral closure of $R$ is $\overline R=k[\![x^4,x^3y,x^2y^2,xy^3,y^4]\!]=S^{(4)}$, which is a Veronese subring of $S$ and normal.
As $\overline R/R\cong k$, there is an exact sequence $0\to R\to\overline R\to k\to0$ of $R$-modules.
Hence each $R$-module $M$ admits an exact sequence $0\to V\to M\to M\otimes_R\overline R\to W\to 0$ with $V,W$ being $k$-vector spaces, and $[M]=[M\otimes_R\overline R]$ in $\h(R)$.
Letting $S^i$ be the $\overline R$-submodule of $S$ generated by monomials of degree $i$, we have $S=S^0\oplus S^1\oplus S^2\oplus S^3$.
The equalities $S^1=3S^3$ and $S^2=2S^3$ hold in $\Cl(\overline R)$, and $\Cl(\overline R)=\Z S^3\cong\Z/4\Z$ (see \cite[Theorem 2.3.1]{BG} and \cite[Lemma (4.1)]{K}).
Hence $\h(\overline R)\cong\Z\oplus\Z/4\Z$ and $\h(\overline R)_\R=\R[\overline R]$.
Proposition \ref{39}(2) implies that $\h(R)_\R$ is a homomorphic image of $\h(\overline R)_\R$.
It follows that $\dim_\R\h(R)_\R\le 1<\infty$.
\item
We can also construct an example of a Cohen--Macaulay non-normal local domain $R$ of dimension more than two with $\dim_\R\h(R)_\R<\infty$.
Let $T=k[\![s,t,u]\!]$ be a formal power series ring, and consider the two subrings $S=k[\![s^2,st,su,tu]\!]$ and $R=k[\![s^4,st,su,tu]\!]$.
Then $R\subseteq S=R\cdot1+R\cdot s^2\subseteq T$.
There is an isomorphism $S\cong k[\![x,y,z,w]\!]/(xw-yz)$, which induces $R\cong k[\![v,y,z,w]\!]/(vw^2-y^2z^2)=:A$.
Note that $(y,w)A$ belongs to the singular locus of $A$ and has height $1$.
Thus $R$ is not normal.
We have $S/R\cong R/\ann_R(\overline{s^2})$, and $\ann_R(\overline{s^2})$ contains $(st)(su)$ and $tu$.
Hence $S/R$ is a module over the ring $R/I\cong k[\![v,y,z]\!]/(yz)$, where $I:=((st)(su),tu)R$.
We have $\Cl(S)\cong\Z$ by \cite[Proposition 14.8]{F}, which shows $\h(S)\cong\Z^{\oplus2}$.
Proposition \ref{39}(2) shows $\dim_\R\h(R)_\R\le\dim_\R\h(S)_\R=2<\infty$.
\end{enumerate}
\end{ex}

The proposition below gives some information on neiborhoods of points in $\h(R)_\R$ determined by modules.
For each $x\in\h(R)_\R$, we denote by $\b_r(x)$ the open ball in $\h(R)_\R$ of radius $r$ centered at $x$.

\begin{prop}\label{16}
Let $R$ be a domain such that $\h(R)$ is finitely generated.
Let $M$ be an $R$-module.
Then $[M]$ is the only point in $\b_1([M])$ that has the form $[N]$ with $N$ an $R$-module.
\end{prop}

\begin{proof}
Since the $\Z$-module $\h(R)$ is finitely generated, we have $\h(R)=\Z[R]+\sum_{i=1}^\zeta\Z[\Phi_i]+\sum_{j=1}^\eta\Z[\Psi_j]$.
Let $N$ be an $R$-module such that $[N]\in\b_1([M])$.
Then there are real numbers $c_0,c_1,\dots,c_\zeta$ with $|c_i|<1$ for all $i$ such that the equality $[N]-[M]=c_0[R]+\sum_{i=1}^\zeta c_i[\Phi_i]$ holds in $\h(R)_\R$.
Sending this by $\rk_\R$, we get an equality $\rk N-\rk M=c_0$ in $\R$.
Hence $c_0$ is an integer, and $c_0=0$ as $-1<c_0<1$.
Put $r:=\rk M=\rk N$.
There are exact sequences $0\to R^{\oplus r}\to M\to A\to0$ and $0\to R^{\oplus r}\to N\to B\to0$ with $A,B$ torsion.
Thus $[A],[B]\in\k(R)$.
We can write $[A]=\sum_{i=1}^\zeta a_i[\Phi_i]+p$ and $[B]=\sum_{i=1}^\zeta b_i[\Phi_i]+q$ in $\h(R)$, where $a_i,b_i\in\Z$ and $p,q\in\sum_{j=1}^\eta\Z[\Psi_j]$.
There are equalities $[M]=r[R]+[A]$ and $[N]=r[R]+[B]$ in $\h(R)$, while $[A]=\sum_{i=1}^\zeta a_i[\Phi_i]$ and $[B]=\sum_{i=1}^\zeta b_i[\Phi_i]$ in $\h(R)_\R$ since $\theta_j[\Psi_j]=0$ for every $1\le j\le\eta$.
We obtain equalities $\sum_{i=1}^\zeta(b_i-a_i)[\Phi_i]=[B]-[A]=[N]-[M]=\sum_{i=1}^\zeta c_i[\Phi_i]$ in $\h(R)_\R$, which implies $b_i-a_i=c_i$ for all $1\le i\le \zeta$.
Since $b_i-a_i$ is an integer and $-1<c_i<1$, we have $c_i=0$ for all $1\le i\le \zeta$.
It follows that $[N]=[M]$, which completes the proof of the proposition.
\end{proof}

\section{The interior, closure and boundary of $\c(R)$}

A subset $C$ of an $\R$-vector space $V$ is called a {\em convex cone} if $ax+by\in C$ for all $x,y\in C$ and $a,b\in\rp$.
In this section, we investigate the structure of the interior, closure and boundary of the convex cone $\c(R)$, whose definition is given below.

\begin{dfn}
We denote by $\c(R)$ the convex cone in $\h(R)_\R$ spanned by maximal Cohen--Macaulay modules, that is,
$$\textstyle
\c(R)=\sum_{M\in\cm(R)}\rp[M]\subseteq\h(R)_\R
$$
\end{dfn}

We recall that the {\em Cohen--Macaulay cone} $\ccm(R)$ of $R$ is defined as the convex cone
$$\textstyle
\ccm(R)=\sum_{M\in\cm(R)}\rp[M]\subseteq\overline{\g(R)}_\R.
$$
Proposition \ref{47} guarantees that it is meaningful to study our convex cone $\c(R)$ to understand the Cohen--Macaulay cone $\ccm(R)$.

Here are some basic properties of the cone $\c(R)$ and its closure $\overline{\c(R)}$ in the space $\h(R)_\R$.

\begin{prop}\label{15-0}
Let $R$ be a domain.
Then the following statements hold.\\
{\rm(1)} $\c(R)\cap\k(R)_\R=\{0\}$.\qquad
{\rm(2)} $\c(R)\subseteq(\RP[R]+\k(R)_\R)\cup\{0\}$.\qquad
{\rm(3)} $\overline{\c(R)}\subseteq\rp[R]+\k(R)_\R$.\\
{\rm(4)} $\overline{\c(R)}$ is a convex cone.
\end{prop}

\begin{proof}
(1) Take any element $x\in\c(R)\cap\k(R)_\R$.
We can write $x=\sum_{i=1}^ua_i[M_i]$, where $a_i$ is a nonnegative real number and $M_i$ is a maximal Cohen--Macaulay $R$-module.
Put $r_i=\rk M_i$ for each $i$.
Using the extended rank map $\rk_\R:\h(R)_\R\to\R$, we get $\rk_\R(x)=\sum_{i=1}^ua_ir_i$ in $\R$.
As $x\in\k(R)_\R=\ker(\rk_\R)$, we have $\rk_\R(x)=0$.
Since $a_i\ge0$ and $r_i>0$ for all $i$, we must have $a_i=0$ for all $i$, and obtain $x=0$.

(2) Let $0\ne x\in\c(R)$.
Write $x=\sum_{i=1}^ua_i[M_i]$, where $u\in\ZP$, $a_i\in\RP$ and $M_i$ is a maximal Cohen--Macaulay $R$-module; set $r_i=\rk M_i$.
There is an exact sequence $0\to R^{\oplus r_i}\to M_i\to T_i\to0$ with $T_i$ torsion.
We have $[M_i]=r_i[R]+[T_i]$, and get $x=\sum_{i=1}^ua_i[M_i]=\sum_{i=1}^ua_i(r_i[R]+[T_i])=(\sum_{i=1}^ua_ir_i)[R]+\sum_{i=1}^ua_i[T_i]$.
Since $\sum_{i=1}^ua_ir_i\in\RP$ and $\sum_{i=1}^ua_i[T_i]\in\k(R)_\R$, the point $x$ is in $\RP[R]+\k(R)_\R$.

(3) Note that $(\RP[R]+\k(R)_\R)\cup\{0\}$ is contained in $\rp[R]+\k(R)_\R$.
It is easy to observe that $\rp[R]+\k(R)_\R$ is a closed subset of $\h(R)_\R$.
The assertion follows from (2).

(4) Let $p,q\in\overline{\c(R)}$ and $a,b\in\rp$.
Then $p=\lim_{i\to\infty}p_i$ and $q=\lim_{i\to\infty}q_i$ for some $p_i,q_i\in\c(R)$.
As $\c(R)$ is a convex cone, $ap_i+bq_i\in\c(R)$ for each $i$.
Hence $ap+bq=\lim_{i\to\infty}(ap_i+bq_i)\in\overline{\c(R)}$.
\end{proof}

We investigate the maps $\vv_\omega$ and $(\vv_\omega)_\R$, where $\omega$ is a canonical module of a Cohen--Macaulay ring.

\begin{lem}\label{62}
Let $R$ be a Cohen--Macaulay normal domain with a canonical module $\omega$.
\begin{enumerate}[\rm(1)]
\item
The endomorphism $\vv_\omega:\h(R)\to\h(R)$ is surjective.
\item
One has $\vv_\omega([M])=[M^\dag]$ and $\vv_\omega^2([M])=[M]$ for each maximal Cohen--Macaulay $R$-module $M$.
\item
Suppose that $\h(R)_\R$ is a finite-dimensional $\R$-vector space.
\begin{enumerate}[\rm(a)]
\item
The linear transformation $(\vv_\omega)_\R:\h(R)_\R\to\h(R)_\R$ is a homeomorphism.
\item
One has $(\vv_\omega)_\R(\c(R))=\c(R)$, $(\vv_\omega)_\R(\overline{\c(R)})=\overline{\c(R)}$ and $(\vv_\omega)_\R(\int\c(R))=\int\c(R)$.
\end{enumerate}
\end{enumerate}
\end{lem}

\begin{proof}
(1) For each $R$-module $M$ of rank $r$ there is a torsion $R$-module $C$ with $[M]=r[R]+[C]$ in $\g(R)$.
We have $[R]=\vv_\omega([\omega])$ and $[C]=\vv_\omega(-[C])$, whence $\vv_\omega$ is surjective.

(2) The assertion is straightforward.

(3a) By (1) the linear transformation $(\vv_\omega)_\R$ is surjective, and hence it is an automorphism.
In general, a linear map of finite-dimensional $\R$-vector spaces is continuous.
Thus $(\vv_\omega)_\R$ is a homeomorphism.

(3b) The assertion follows from (2) and (3a).
\end{proof}

Recall that a convex cone $C$ is called {\em strongly convex} if $C\cap-C=\{0\}$.
The following result gives some information on the shapes of the convex cones $\c(R)$ and $\overline{\c(R)}$.

\begin{prop}\label{48}
Let $R$ be a Cohen--Macaulay normal domain with a canonical module $\omega$.
Suppose that $\h(R)_\R$ is a finite-dimensional $\R$-vector space.
Then the following statements hold.
\begin{enumerate}[\rm(1)]
\item
One has that $\overline{\c(R)}\cap\k(R)_\R$ is a subspace of the $\R$-vector space $\k(R)_\R$.
\item
There is an equality $\overline{\c(R)}\cap-\overline{\c(R)}=\overline{\c(R)}\cap\k(R)_\R$.
In particular, $\overline{\c(R)}$ is a strongly convex cone if and only if $\overline{\c(R)}\cap\k(R)_\R=0$.
\item
Let $x=x_0[R]+\sum_{i=1}^\zeta x_i[\Phi_i]$ and $y=x_0[\omega]-\sum_{i=1}^\zeta x_i[\Phi_i]$ be points in $\h(R)_\R$.\\
{\rm(a)} $x\in\c(R)$ if and only if $y\in\c(R)$.\qquad
{\rm(b)} $x\in\overline{\c(R)}$ if and only if $y\in\overline{\c(R)}$.
\item
If $R$ is Gorenstein, then both $\c(R)$ and $\overline{\c(R)}$ are symmetric with respect to the axis $\R[R]$.
\end{enumerate}
\end{prop}

\begin{proof}
Set $f=(\vv_\omega)_\R$.
In what follows, we tacitly use Lemma \ref{62}.

(1) Combining Proposition \ref{15-0}(4) with the fact that $\k(R)_\R$ is an $\R$-vector space, we see that $\overline{\c(R)}\cap\k(R)_\R$ is closed under sums and scalar multiplication by a nonnegative real number.
It is thus enough to verify that for each $x\in\overline{\c(R)}\cap\k(R)_\R$ it holds that $-x\in\overline{\c(R)}\cap\k(R)_\R$.
We have $f(\overline{\c(R)}\cap\k(R)_\R)=f(\overline{\c(R)})\cap f(\k(R)_\R)=\overline{\c(R)}\cap\k(R)_\R$, and hence $-x=f(x)\in\overline{\c(R)}\cap\k(R)_\R$.

(2) By (1) the set $\overline{\c(R)}\cap\k(R)_\R$ is closed under multiplication by $-1$.
This implies that $\overline{\c(R)}\cap\k(R)_\R$ is contained in $\overline{\c(R)}\cap-\overline{\c(R)}$.
By Proposition \ref{15-0}(3) the set $\overline{\c(R)}\cap-\overline{\c(R)}$ is contained in $(\rp[R]+\k(R)_\R)\cap(\rmm[R]+\k(R)_\R)=\k(R)_\R$.
Hence $\overline{\c(R)}\cap-\overline{\c(R)}$ is contained in $\overline{\c(R)}\cap\k(R)_\R$.

(3) Note that $f(x)=y$.
The assertion follows from the equalities $f(\c(R))=\c(R)$ and $f(\overline{\c(R)})=\overline{\c(R)}$.

(4) Since $R$ is a Gorenstein ring, we have $\omega=R$.
The lines $\R[R],\,\R[\Phi_1],\,\dots,\,\R[\Phi_\zeta]$ form the axes of the $(\zeta+1)$-dimensional space $\h(R)_\R$.
Thus the assertion immediately follows from (3).
\end{proof}

To explore inner points of $\c(R)$, we make a remark.

\begin{rem}\label{41}
Suppose that $x_1,\dots,x_n\in\c(R)$ form a basis of the ambient $\R$-vector space $\h(R)_\R$.
Let $a_1,\dots,a_n\in\RP$ and put $x=\sum_{i=1}^na_ix_i$.
Then $x$ is an inner point of $\c(R)$.
In fact, letting $r$ be the minimum of the distances between $x$ and the faces of the convex cone $\sum_{i=1}^n\rp x_i$, we have $\b_r(x)\subseteq\c(R)$.
We refer the reader to \cite[Exercise 3.1]{Br}.
\end{rem}

The proposition below states some properties of the interiors of convex cones.
Assertion (3a) corresponds to the result \cite[Lemma 2.5(7)]{CK} concerning the Cohen--Macaulay cone $\ccm(R)$.

\begin{prop}\label{35}
Suppose that $\h(R)_\R$ is a finite-dimensional $\R$-vector space.
\begin{enumerate}[\rm(1)]
\item
Let $V$ be a convex cone in $\h(R)_\R$.
Then the following three statements hold true.\\
{\rm(a)} If $x\in\int V$ and $y\in V$, then $x+y\in\int V$.\qquad
{\rm(b)} If $x\in\int V$ and $a\in\RP$, then $ax\in\int V$.\\
{\rm(c)} Let $x,y\in V$.
Then one has the equivalences below.
$$
x\in\int V\iff
x-\varepsilon y\in\int V\text{ for some }\varepsilon\in\RP\iff
x-\textstyle\frac{1}{n}y\in\int V\text{ for some }n\in\ZP.
$$
\item
Suppose that $R$ is a Cohen--Macaulay local domain.
Then the following two statements hold true.\\
{\rm(a)} One has $[R]\in\int\c(R)$.\qquad
{\rm(b)} If $R$ is normal and has a canonical module $\omega$, then $[\omega]\in\int\c(R)$.
\end{enumerate}
\end{prop}

\begin{proof}
(1) As $x$ is an inner point of $ V$, there exists $\varepsilon\in\RP$ such that $\b_\varepsilon(x)$ is contained in $ V$.

(a) For each $z\in\b_\varepsilon(x+y)$, we have $\varepsilon>\|z-(x+y)\|=\|(z-y)-x\|$, which implies $z-y\in\b_\varepsilon(x)$.
Hence $z=(z-y)+y\in V$.
Thus $\b_\varepsilon(x+y)\subseteq V$, which shows that $x+y$ is an inner point of $ V$.

(b) For any $z\in\b_{a\varepsilon}(ax)$, we have $\|z-ax\|<a\varepsilon$.
Hence $\|\frac{z}{a}-x\|<\varepsilon$, and we get $\frac{z}{a}\in\b_{\varepsilon}(x)$.
Therefore $\frac{z}{a}$ belongs to $ V$, and so does $z=a\cdot\frac{z}{a}$.
Thus $\b_{a\varepsilon}(ax)$ is contained in $ V$, which deduces the assertion.

(2) Call the three conditions (a), (b) and (c) in order.
It is clear that (c) implies (b), while (b) implies (a) by (1). 
Assume that (a) holds, and let us deduce (c).
There is nothing to show if $y=0$.
Let $y\ne0$.
As $x\in\int V$, we have $\b_\delta(x)\subseteq V$ for some $\delta\in\RP$.
Choose an integer $n>0$ such that $\frac{1}{n}<\frac{\delta}{\|y\|}$.
Then $\|(x-\frac{1}{n}y)-x\|=\frac{1}{n}\|y\|<\delta$, and it follows that $x-\frac{1}{n}y\in\b_\delta(x)\subseteq V$.
This shows that $x-\frac{1}{n}y$ is an inner point of $ V$, and (c) follows.

(3) First of all, we consider the case $\zeta=0$ (recall that we adopt Setup \ref{20}).
Then $\h(R)_\R=\R[R]$ and $\c(R)=\rp[R]$.
It is clear that $[R]$ is an inner point of $\c(R)$.
If $\omega$ is a canonical module of $R$, then $[\omega]=(\rk_R\omega)[R]=[R]$, which is an inner point of $\c(R)$.
Thus we may assume $\zeta>0$.

Choose a nonzero element $x\in R$ that annihilates the torsion $R$-modules $\Phi_i$ for all $1\le i\le\zeta$.
The exact sequence $0\to R\xrightarrow{x}R\to R/(x)\to0$ implies $[R/(x)]=0$ in $\g(R)$.
It is seen that $[\syz_{R/(x)}^jX]=(-1)^j[X]$ for all $X\in\mod R/(x)$ and $j\ge0$.
Let $l_i\ge0$ be such that $T_i:=\syz_{R/(x)}^{l_i}\Phi_i$ is a nonzero maximal Cohen--Macaulay $R/(x)$-module.
Set $T_0=\syz_{R/(x)}(T_1\oplus\cdots\oplus T_\zeta)$.
This is also a maximal Cohen--Macaulay $R/(x)$-module, and we observe $[T_0]+[T_1]+\cdots+[T_\zeta]=0$ in $\g(R)$.
For each $i$ there is an exact sequence
\begin{equation}\label{14}
0\to\syz_RT_i\to R^{\oplus b_i}\to T_i\to0
\end{equation}
of $R$-modules, which gives an equality
\begin{equation}\label{65}
[\syz_RT_i]=b_i[R]-[T_i]
\end{equation}
in $\g(R)$.
Putting $b=\sum_{i=0}^\zeta b_i>0$, we get $\sum_{i=0}^\zeta[\syz_RT_i]=b[R]$ in $\g(R)$.
Hence in $\h(R)_\R$ we have
\begin{equation}\label{66}
\textstyle
[R]=\frac{1}{b}\sum_{i=0}^\zeta[\syz_RT_i].
\end{equation}

We prove that $[\syz_RT_0],\dots,[\syz_RT_\zeta]$ belong to $\c(R)$ and form an $\R$-basis of $\h(R)_\R$.
The former assertion follows from the fact that each $\syz_RT_i$ is a maximal Cohen--Macaulay $R$-module.
To show the latter, it suffices to check that $[\syz_RT_0],\dots,[\syz_RT_\zeta]$ span $\h(R)_\R$.
This is deduced by using \eqref{65}, \eqref{66} and the equalities $[T_i]=(-1)^{l_i}[\Phi_i]$ for $1\le i\le\zeta$.

Now (a) follows from Remark \ref{41}.
Assertion (b) is deduced from assertion (a), by means of (2) and (3a) of Lemma \ref{62}.
\end{proof}

Next we consider boundary points of the convex cone $\c(R)$.
We begin with recalling the definition of a totally reflexive module.

\begin{dfn}
\begin{enumerate}[(1)]
\item
An $R$-module $M$ is called {\em totally reflexive} if $M$ is reflexive and $\Ext_R^{>0}(M\oplus M^\ast,R)=0$.
Total reflexivity is preserved under $(-)^\ast$.
If $R$ is Cohen--Macaulay (resp. Gorenstein), then total reflexivity implies (resp. is equivalent to) maximal Cohen--Macaulayness.
\item
Let $(R,\m,k)$ be a local ring.
Let $M$ be a totally reflexive $R$-module.
Then there uniquely exists an exact sequence $\cdots\xrightarrow{d_2}F_1\xrightarrow{d_1}F_0\xrightarrow{d_0}F_{-1}\xrightarrow{d_{-1}}F_{-2}\xrightarrow{d_{-2}}\cdots$ of free $R$-modules such that $M\cong\im d_0$ and $\im d_i\subseteq\m F_{i-1}$ for all $i\in\Z$.
For each integer $n\ge0$, the $n$-th {\em cosyzygy} $\syz^{-n}M$ of $M$ is defined as the image of the map $d_{-n}$.
Total reflexivity is preserved under $\syz^i$ for each $i\in\Z$.
Note that $\syz^{-n}M\cong(\syz^n(M^\ast))^\ast$.
We set $\beta_{-n}^R(M)=\rk_RF_{-n}$, so that $\beta_i^R(M)=\rk_RF_i$ for all $i\in\Z$.
\end{enumerate}
\noindent
For the details of totally reflexive modules, we refer the reader to \cite{C}.
\end{dfn}

Now we can state and prove the following result on the boundary of $\c(R)$.

\begin{prop}\label{44}
Let $R$ be a Cohen--Macaulay normal local ring such that $\h(R)_\R$ is a finite-dimensional $\R$-vector space.
Let $M$ be a totally reflexive $R$-module.
Suppose that either $[M]$ or $[M^\ast]$ is a boundary point of the convex cone $\c(R)$.
Then the inequality $\rk_R(\syz^iM)\ge\rk_RM$ holds for all $i\in\Z$.
\end{prop}

\begin{proof}
Put $r=\rk M$ and $b_i=\beta_i(M)$ for each $i\in\Z$.
Let $n\ge0$ be an integer.
The exact sequence
\begin{equation}\label{43}
0\to\syz^{2n}M\to R^{\oplus b_{2n-1}}\to R^{\oplus b_{2n-2}}\to\cdots\to R^{\oplus b_1}\to R^{\oplus b_0}\to M\to0
\end{equation}
shows $[M]=(b_0-b_1+\cdots+b_{2n-2}-b_{2n-1})[R]+[\syz^{2n}M]=(r-\rk\syz^{2n}M)[R]+[\syz^{2n}M]$ in $\g(R)$.
Since $R$ and $\syz^{2n}M$ are maximal Cohen--Macaulay $R$-modules, the points $[R]$ and $[\syz^{2n}M]$ are in $\c(R)$.
In view of (1a), (1b) and (2a) of Proposition \ref{35}, we must have $r-\rk\syz^{2n}M\le0$ if $[M]\in\partial\c(R)$.

From \eqref{43} we get an exact sequence $0\to M^\ast\to R^{\oplus b_0}\to\cdots\to R^{\oplus b_{2n-1}}\to(\syz^{2n}M)^\ast\to0$, which implies $[M^\ast]=(b_0-b_1+\cdots+b_{2n-2}-b_{2n-1})[R]+[(\syz^{2n}M)^\ast]=(r-\rk\syz^{2n}M)[R]+[(\syz^{2n}M)^\ast]$ in $\g(R)$.
Similarly as above, we must have $r-\rk\syz^{2n}M\le0$ if $[M^\ast]\in\partial\c(R)$.

Proposition \ref{21} yields $[M]+[M^\ast]=2r[R]$ in $\h(R)$.
The exact sequence
\begin{equation}\label{45}
0\to\syz^{2n+1}M\to R^{\oplus b_{2n}}\to R^{\oplus b_{2n-1}}\to\cdots\to R^{\oplus b_1}\to R^{\oplus b_0}\to M\to0
\end{equation}
shows $[M]=(b_0-b_1+\cdots-b_{2n-1}+b_{2n})[R]-[\syz^{2n+1}M]=(r+\rk\syz^{2n+1}M)[R]-[\syz^{2n+1}M]$ in $\g(R)$.
We obtain $[M^\ast]=2r[R]-[M]=(r-\rk\syz^{2n+1}M)[R]+[\syz^{2n+1}M]$ in $\h(R)$.
In a similar way as above, we must have $r-\rk\syz^{2n+1}M\le0$ if $[M^\ast]\in\partial\c(R)$.

From \eqref{45} we get an exact sequence $0\to M^\ast\to R^{\oplus b_0}\to\cdots\to R^{\oplus b_{2n}}\to(\syz^{2n+1}M)^\ast\to0$.
This gives $[M^\ast]=(b_0-b_1+\cdots-b_{2n-1}+b_{2n})[R]-[(\syz^{2n+1}M)^\ast]=(r+\rk\syz^{2n+1}M)[R]-[(\syz^{2n+1}M)^\ast]$ in $\g(R)$, and hence $[M]=2r[R]-[M^\ast]=(r-\rk\syz^{2n+1}M)[R]+[(\syz^{2n+1}M)^\ast]$ in $\h(R)$.
Similarly as above, we must have $r-\rk\syz^{2n+1}M\le0$ if $[M]\in\partial\c(R)$.

Now we conclude that if either $[M]$ or $[M^\ast]$ belongs to $\partial\c(R)$, then $\rk\syz^iM\ge r$ for all $i\ge0$.
By symmetry, we observe that if either $[M]$ or $[M^\ast]$ belongs to $\partial\c(R)$, then $\rk\syz^i(M^\ast)\ge r$ for all $i\ge0$.
Note that $\rk\syz^{-i}M=\rk(\syz^i(M^\ast))^\ast=\rk\syz^i(M^\ast)$.
Therefore, $\rk\syz^iM\ge r$ for all $i\in\Z$.
\end{proof}

\section{Convex polyhedral subcones of $\c(R)$}

Recall that a convex cone $C$ is called {\em polyhedral} if there exist a finite number of vectors $v_1,\dots,v_n$ such that $C=\sum_{i=1}^n\rp v_i$.
In this section, we explore convex polyhedral subcones of $\c(R)$, that is, convex cones in $\h(R)_\R$ spanned by a finite number of maximal Cohen--Macaulay modules.
More precisely, we try to construct as big a convex polyhedral subcone of $\c(R)$ as possible.

An obvious example where $\c(R)$ is itself polyhedral is given by a ring of finite Cohen--Macaulay representation type.
According to \cite[Paragraph following Definition 2.4]{CK}, it is unknown if $\ccm(R)$ is polyhedral or not.
In view of Proposition \ref{47}, it is interesting to think about the question asking whether the convex cone $\c(R)$ is polyhedral or not.
Unfortunately, we cannot give an answer to the question itself, but prove Theorem \ref{11} below, which relates with the question in the affirmative direction.

Let $R$ be a local ring, and let $M$ be an $R$-module.
The {\em complexity} of $M$ is defined by
$$
\cx_RM=\inf\{n\in\zp\mid\text{there exists $\alpha\in\rp$ such that $\beta_i^R(M)\le\alpha\cdot i^{n-1}$ for all $i\gg0$}\}.
$$
We denote by $\e(M)$ the (Hilbert--Samuel) multiplicity of $M$.
Recall that for every Cohen--Macaulay local ring $R$ it holds that $\e(R)\ge\codim R-1$, and that $R$ is said to have {\em minimal multiplicity} if the equality holds (see \cite[Exercise 4.6.14]{BH}).
The following theorem is the main result of this section.

\begin{thm}\label{11}
Let $R\subseteq S$ be a finite extension of Cohen--Macaulay local domains.
Suppose $R$ is normal, $\h(S)_\R$ is finite-dimensional, and $\c(S)$ is polyhedral.
Let $f:\h(S)_\R\to\h(R)_\R$ be the natural surjection\footnote{the map defined as in Proposition \ref{39}(2)}, and set $V=f(\c(S))\subseteq\c(R)$.
Let $M$ be an $R$-module with $\cx_RM\ge2$ which is locally free in codimension one.
Then $[\syz_R^nM]\in V$ for infinitely many $n\in\NN$, if either of the following two conditions holds.
\begin{enumerate}[\rm(1)]
\item
There exists an $S$-module $N$, which is as an $R$-module locally free in codimension one, such that $[L]\in\int\c(S)$, where $L:=N\oplus N^\ast$.
\item
The module $S^\ast$ is maximal Cohen--Macaulay.
\end{enumerate}
\end{thm}

\begin{proof}
(1) As $f$ is a surjective linear map of finite-dimensional $\R$-vector spaces, it is an open map (by base change and openness of projection).
As high enough syzygies are maximal Cohen--Macaulay, we see that $\dim\c(S)=\dim_\R\h(S)_\R$.
It is observed by the surjectivity of $f$ that $\dim V=\dim_\R\h(R)_\R$.
Therefore $f(\int\c(S))\subseteq\int f(\c(S))=\int V$, and $[L]=f([L])\in\h(R)_\R$ is an inner point of $V$.
Proposition \ref{21} implies $[L]=[N]+[N^\ast]=2r[R]$ in $\h(R)_\R$.
Hence $[R]$ is an inner point of $V$.
Put $r_n=\rk_R\syz_R^nM$ for each $n\in\NN$.
Write $[M]=r_0[R]+[C]$, where $C$ is a torsion $R$-module.
For each $n\in\NN$ there is an equality
$$
r_n^{-1}[\syz^n_RM]=[R]+(-1)^nr_n^{-1}[C]
$$
in $\h(R)_\R$.
Since $\cx_RM\ge2$, the sequence $\{r_n\}_{n\in\NN}$ is unbounded.
It follows from (1a) and (1c) of Proposition \ref{35} that $r_n^{-1}[\syz^n_RM]$ belongs to $V$ for infinitely many $n\in\NN$, and $[\syz^n_RM]\in V$ for such $n\in\NN$.

(2) Put $N=S$ and $L=N\oplus N^\ast$.
Let $\p$ be a prime ideal of $R$ with codimension one.
As $S$ is a maximal Cohen--Macaulay $R$-module, $N_\p=S_\p$ is a maximal Cohen--Macaulay $R_\p$-module.
Since $R_\p$ is a discrete valuation ring, $N_\p$ is free as an $R_\p$-module.
It follows from (2a) and (1a) of Proposition \ref{35} that $[L]=[S]+[S^\ast]$ is an inner point of $\c(S)$.
Thus the assertion is a consequence of (1).
\end{proof}

\begin{rem}\label{46}
\begin{enumerate}[(1)]
\item
The proof of Theorem \ref{11}(2) says that the assumption in this statement that $S^\ast$ is maximal Cohen--Macaulay can be weakened to the condition that $[S^\ast]\in\c(S)$.
\item
The assumption in Theorem \ref{11}(2) that $S^\ast$ is maximal Cohen--Macaulay is satisfied if $R$ is Gorenstein.
In fact, in this case the $S$-module $S^\ast$ is a canonical module of the Cohen--Macaulay local ring $S$.
\item
The assmption in Theorem \ref{11}(1) that $[L]\in\int\c(S)$ is satisfied if $R,S$ are Gorenstein normal rings and $N$ is as an $S$-module locally free in codimension one.
Indeed, then, we have $\Hom_S(N,S)\cong\Hom_S(N,\Hom_R(S,R))\cong\Hom_R(N,R)$.
Proposition \ref{21} implies that $[L]=[N]+[\Hom_S(N,S)]=2(\rk_SN)[S]$ in $\h(S)$.
Assertions (1b) and (2a) of Proposition \ref{35} yield $[L]\in\int\c(S)$.
\item
The assumption in Theorem \ref{11} that $\cx_RM\ge2$ is satisfied if $\pd_RM=\infty$ and $R$ is not a hypersurface but has minimal multiplicity.
In fact, in this case, we have $\edim R-\depth R=\codim R\ge2$.
We obtain $\cx_RM=\infty$ by \cite[Example 5.2.8 and Theorem 5.3.3(2)]{A}.
\end{enumerate}
\end{rem}

The {\em simple singularities}, also known as ADE singularities, are by definition those hypersurfaces\footnote{To be precise, by simple singularities of dimension $d>0$ (resp. dimension $0$) we mean the quotients of a formal power series ring $k[\![x,y,z_2,\dots,z_d]\!]$ (resp. a formal power series ring $k[\![x]\!]$) over an arbitrary field $k$ by the five kinds of polynomials given in \cite[Theorem (8.8)]{Y} (resp. the polynomials $x^n$ with $n>0$).} that appear in \cite[Theorem (8.8)]{Y}.
Restricting the above theorem to the case where $S$ is a simple singularity, we obtain the following result.

\begin{cor}\label{40}
\begin{enumerate}[\rm(1)]
\item
Let $R$ be a Gorenstein normal local ring possessing a simple singularity as a finite extension.
Then $\c(R)$ has a convex polyhedral subcone $V$ such that for all $R$-modules $M$ with $\cx_RM\ge2$ there exist infinitely many integers $n\ge0$ with $[\syz^nM]\in V$.
\item
Let $R$ be a Cohen--Macaulay non-hypersurface normal local ring with minimal multiplicity.
Suppose that $R$ has a simple singularity as a finite extension whose $R$-dual is maximal Cohen--Macaulay.
Then $\c(R)$ has a convex polyhedral subcone $V$ such that for all $R$-modules $M$ there exist infinitely many integers $n\ge0$ with $[\syz^nM]\in V$.
\end{enumerate}
\end{cor}

\begin{proof}
Let $S$ be the simple singularity, and put $d=\dim R=\dim S$.
If $d\le1$, then $R$ is regular and $\c(R)=\rp[R]$ is itself polyhedral.
We may assume $d\ge2$, so that $S$ is a normal domain.
It follows from \cite[Proposition (13.10)]{Y} that $\g(S)$ is finitely generated, which particularly says $\dim_\R\h(S)_\R<\infty$.
There are only finitely many nonisomorphic indecomposable maximal Cohen--Macaulay $S$-modules by \cite[Corollary (12.6)]{Y}; let $X_1,\dots,X_t$ be those modules.
Then $\c(S)=\sum_{i=1}^t\rp[X_i]$, which is polyhedral.
Put $V=\sum_{i=1}^t\rp[X_i]\subseteq\c(R)$, which is polyhedral as well.

(1) The assertion follows from Theorem \ref{11}(2) and Remark \ref{46}(2).

(2) Fix an $R$-module $M$.
We may assume $\pd_RM=\infty$.
The assertion follows from Theorem \ref{11}(2) and Remark \ref{46}(4).
\end{proof}

\begin{ex}
Let $B=k[x,y,z,w]/(xw-yz)$ be a homogeneous algebra over a field $k$.
Let $S=k[\![x,y,z,w]\!]/(xw-yz)$ be the completion of $B$ (with respect to the irrelevant maximal ideal of $B$).
Then $S$ is a simple singularity of type $(\A_1)$ with dimension $3$.
Assume $\ch k\ne2$.
\begin{enumerate}[(1)]
\item
Let $A=B^{(2)}=k[x^2,xy,xz,xw,y^2,yw,z^2,zw,w^2]$ be the second Veronese subring of $B$, and let $R$ be its completion.
Then it holds that $A=B^G$, where $G=\left\langle\left(\begin{smallmatrix}-1&0&0&0\\0&-1&0&0\\0&0&-1&0\\0&0&0&-1\end{smallmatrix}\right)\right\rangle$.
Since $S$ is normal, so is $R$ by \cite[Proposition 6.4.1]{BH}.
The a-invariant of $B$ is $-2$.
It follows from \cite[Exercise 3.6.21(e)]{BH} that $A$ is Gorenstein, and so is $R$.
Corollary \ref{40}(1) yields a convex polyhedral subcone $V$ of $\c(R)$ such that for all $R$-modules $M$ with $\cx_RM\ge2$ there exist infinitely many integers $n\ge0$ with $[\syz_R^nM]\in V$.
\item
Consider the subring $A=k[x^2,xy,y^2,z,w]$ of $B$, and let $R$ be its completion.
We have $A=B^G$, where $G=\left\langle\left(\begin{smallmatrix}-1&0&0&0\\0&-1&0&0\\0&0&1&0\\0&0&0&1\end{smallmatrix}\right)\right\rangle$.
It follows from \cite[Proposition 6.4.1 and Corollary 6.4.6]{BH} that $A$ is normal and Cohen--Macaulay, and so is $R$.
We have $R\cong k[\![p,q,r,s,t]\!]/\I_2\left(\begin{smallmatrix}p&q&s\\q&r&t\end{smallmatrix}\right)$, where the denominator means the ideal generated by the $2$-minors of the matrix.
It is observed that $Q=(p,r-s,t)$ is a parameter ideal of $R$ with $\m^2=Q\m$, where $\m$ is the maximal ideal of $R$.
Thus $R$ has minimal multiplicity.
We have $S=R1\oplus(Rx+Ry)\cong R\oplus I$ and $I/(q)\cong R/\p$, where $I=(p,q)$ and $\p=(q,r,t)$ are ideals of $R$.
There is an exact sequence $0\to R\xrightarrow{f}I\to R/\p\to0$, where the map $f$ is defined by $f(1)=q$.
This induces an exact sequence $0\to I^\ast\to R\to\Ext_R^1(R/\p,R)$.
We have $\Ext_R^1(R/\p,R)\cong\Hom_{R/(q)}(R/\p,R/(q))\cong(0:_{R/(q)}\p/(q))=I/(q)\cong R/\p$.
We see $I^\ast\cong\p$.
Note that $R/\p\cong k[\![p,s]\!]$.
The depth lemma shows that $\p$ is a maximal Cohen--Macaulay $R$-module, and so is $I^\ast$, and so is $S^\ast$.
Corollary \ref{40}(2) provides a convex polyhedral subcone $V$ of $\c(R)$ satisfying the condition that for all $R$-modules $M$ there exist infinitely many $n\ge0$ such that $[\syz_R^nM]\in V$ holds\footnote{This is itself an immediate consequence of the fact that $R$ is one of the two known examples of a non-Gorenstein Cohen--Macaulay local ring of finite Cohen--Macaulay representation type (at least when $k$ is an algebraically closed field of characteristic zero); see \cite[Proposition (16.12)]{Y}.}.
\end{enumerate}
\end{ex}

\section{The convex cone $\c_r(R)$}

In this section, we introduce a convex cone spanned by maximal Cohen--Macaulay modules of fixed rank, and explore topological properties characterizing finiteness of the number of those maximal Cohen--Macaulay modules.
We begin with giving the precise definition of the convex cone.

\begin{dfn}
For each $r\in\ZP$, we denote by $\c_r(R)$ the convex cone spanned by maximal Cohen--Macaulay $R$-modules of rank $r$, that is, it is defined by the following.
$$\textstyle
\c_r(R)=\sum_{M\in\cm(R),\,\rk M=r}\rp[M]\subseteq\h(R)_\R.
$$
Note that there is an equality $\c(R)=\sum_{r\in\ZP}\c_r(R)$.
\end{dfn}

The same types of assertions as in Propositions \ref{15-0} and \ref{48} hold true for the convex cone $\c_r(R)$.

\begin{prop}\label{15}
Let $R$ be a domain and $r>0$ an integer.
Then the following hold.\\
{\rm(1)} $\c_r(R)\cap\k(R)_\R=\{0\}$.\qquad
{\rm(2)} $\c_r(R)\subseteq(\RP[R]+\k(R)_\R)\cup\{0\}$.\qquad
{\rm(3)} $\overline{\c_r(R)}\subseteq\rp[R]+\k(R)_\R$.\\
{\rm(4)} $\overline{\c_r(R)}$ is a convex cone.
\end{prop}

\begin{proof}
Assertions (1), (2) and (3) follow from assertions (1), (2) and (3) of Proposition \ref{15-0}, since $\c_r(R)$ is contained in $\c(R)$.
Assertion (4) is similarly shown to assertion (4) of Proposition \ref{15-0}.
\end{proof}

\begin{prop}\label{1p}
Let $R$ be a Cohen--Macaulay normal domain with a canonical module $\omega$.
Suppose that $\h(R)_\R$ is a finite-dimensional $\R$-vector space.
Then the following statements hold for each integer $r>0$.
\begin{enumerate}[\rm(1)]
\item
One has that $\overline{\c_r(R)}\cap\k(R)_\R$ is a subspace of the $\R$-vector space $\k(R)_\R$.
\item
There is an equality $\overline{\c_r(R)}\cap-\overline{\c_r(R)}=\overline{\c_r(R)}\cap\k(R)_\R$.
In particular, $\overline{\c_r(R)}$ is a strongly convex cone if and only if $\overline{\c_r(R)}\cap\k(R)_\R=0$.
\item
Let $x=x_0[R]+\sum_{i=1}^\zeta x_i[\Phi_i]$ and $y=x_0[\omega]-\sum_{i=1}^\zeta x_i[\Phi_i]$ be points in $\h(R)_\R$.\\
{\rm(a)} $x\in\c_r(R)$ if and only if $y\in\c_r(R)$.\qquad
{\rm(b)} $x\in\overline{\c_r(R)}$ if and only if $y\in\overline{\c_r(R)}$.
\item
If $R$ is Gorenstein, then both $\c_r(R)$ and $\overline{\c_r(R)}$ are symmetric with respect to the axis $\R[R]$.
\end{enumerate}
\end{prop}

\begin{proof}
The arguments along the same lines as in the proof of Proposition \ref{48} work.
Note that if $M$ is a maximal Cohen--Macaulay $R$-module of rank $r$, then so is $M^\dag$.
\end{proof}

We say that an element $p$ of $\h(R)$ (resp. $\h(R)_\R)$ is a {\em  (rank $s$) maximal Cohen--Macaulay point} if there exists a (rank $s$) maximal Cohen--Macaulay $R$-module $M$ such that $p=[M]$ in $\h(R)$ (resp. $\h(R)_\R$).
Similarly, $p\in\Cl(R)$ (resp. $\Cl(R)_\R)$ is called a {\em maximal Cohen--Macaulay point} if there is a maximal Cohen--Macaulay reflexive $R$-module $M$ of rank one such that $p=M$ in $\Cl(R)$ (resp. $\Cl(R)_\R$).
We denote by $\c_r^0(R)$ the set of points $x\in\c_r(R)$ with $\rk_\R(x)=r$.
If $\dim_\R\h(R)_\R<\infty$, then $\c_r^0(R)$ is the set of points $x\in\c_r(R)$ having the form $r[R]+\sum_{i=1}^\zeta c_i[\Phi_i]$ with $c_i\in\R$.
The theorem below is the main result of this section.

\begin{thm}\label{1}
Let $R$ be a domain with $\dim_\R\h(R)_\R<\infty$.
Let $r\in\ZP$.
Consider the conditions below.\\
\qquad
{\rm(a)} There are only finitely many nonisomorphic maximal Cohen--Macaulay $R$-modules of rank $r$.\\
\qquad
{\rm(b)} There are only finitely many rank $r$ maximal Cohen--Macaulay points in $\c_r(R)$.\\
\qquad
{\rm(c)} The convex cone $\c_r(R)$ is polyhedral.\qquad
{\rm(d)} The subset $\c_r(R)$ of $\h(R)_\R$ is closed.\\
\qquad
{\rm(e)} The equality $\overline{\c_r(R)}\cap\k(R)_\R=0$ holds.\qquad
{\rm(f)} The subset $\c_r^0(R)$ of $\h(R)_\R$ is bounded.\\
\qquad
{\rm(g)} The convex cone $\overline{\c_r(R)}$ is strongly convex.\\
Then the implications {\rm(a)} $\Rightarrow$ {\rm(b)} $\Rightarrow$ {\rm(c)} $\Rightarrow$ {\rm(d)} $\Rightarrow$ {\rm(e)} $\Leftrightarrow$ {\rm(f)} $\Rightarrow$ {\rm(g)} hold.
If $\h(R)$ is finitely generated, then {\rm(b)}--{\rm(f)} are equivalent.
If $R$ is normal, $\h(R)$ is finitely generated and $r=1$, then {\rm(a)}--{\rm(f)} are equivalent.
If $R$ is a normal Cohen--Macaulay ring with a canonical module, then {\rm(e)} and {\rm(g)} are equivalent.
\end{thm}

\begin{proof}
The last assertion is a repetition of the second assertion of Proposition \ref{1p}(2).

It is evident that the implications (a) $\Rightarrow$ (b) $\Rightarrow$ (c) hold.
Every convex polyhedral cone is a closed subset of the ambient Euclidean space (essentially by Carath\'{e}odory's theorem).
This shows that (c) implies (d).
It follows from Proposition \ref{15}(1) that (d) implies (e).

Assume $\c_r^0(R)$ is unbounded.
Choose a sequence $\{\tau_i\}_{i\in\NN}$ in $\k(R)_\R$ such that $r[R]+\tau_i\in\c_r^0(R)$ and $\delta_i=\|\tau_i\|>i$ for each $i\in\NN$.
Then $\delta_i^{-1}\tau_i$ is in the unit sphere $S=\{v\in\k(R)_\R\mid\|v\|=1\}$ in $\k(R)_\R$.
As $S$ is a bounded closed subset of $\k(R)_\R\cong\R^{\oplus\zeta}$, it is compact.
Since $\R^{\oplus\zeta}$ is a complete metric space, $S$ is a sequentially compact subset of $\k(R)_\R$.
We find a subsequence of $\{\delta_i^{-1}\tau_i\}_{i\in\NN}$ converging to a point $\varsigma\in S$.
Changing indices, we may assume $\lim_{i\to\infty}\delta_i^{-1}\tau_i=\varsigma$.
Then $\lim_{i\to\infty}\delta_i^{-1}(r[R]+\tau_i)=\varsigma\in S\subseteq\k(R)_\R$.
As $\delta_i(r[R]+\tau_i)$ belongs to $\c_r(R)$, we get $0\ne\varsigma\in\overline{\c_r(R)}\cap\k(R)_\R$.
We have shown that (e) implies (f).

Let $0\ne x\in\overline{\c_r(R)}\cap\k(R)_\R$.
We have $x=\lim_{n\to\infty}x_n$ for some sequence $\{x_n\}_{n\in\NN}$ in $\c_r(R)$ such that $x_n\ne0$ for infinitely many $n$.
Changing indices, we may assume $x_n\ne0$ for all $n$.
Write $x=\sum_{i=1}^\zeta x_i[\Phi_i]$ and $x_n=x_{n0}[R]+\sum_{i=1}^\zeta x_{ni}[\Phi_i]$.
As $x_n\in\c_r(R)$, we see that $x_{n0}=\rk_\R(x_n)>0$ (see Proposition \ref{15}(1)).
The point $rx_{n0}^{-1}x_n$ belongs to $\c_r^0(R)$.
We have $\|rx_{n0}^{-1}x_n\|^2=r^2x_{n0}^{-2}\|x_n\|^2=r^2x_{n0}^{-2}(x_{n0}^2+x_{n1}^2+\cdots+x_{n\zeta}^2)=r^2(1+x_{n0}^{-2}(x_{n1}^2+\cdots+x_{n\zeta}^2))$.
Since $\lim_{n\to\infty}x_{n0}=0$ and $\lim_{n\to\infty}x_{ni}=x_i$ for each $1\le i\le\zeta$, we see that $\lim_{n\to\infty}\|rx_{n0}^{-1}x_n\|=\infty$.
Therefore $\c_r^0(R)$ is unbounded.
Thus we have proved that (f) implies (e).

Proposition \ref{15}(3) gives $\overline{\c_r(R)}\cap-\overline{\c_r(R)}\subseteq(\rp[R]+\k(R)_\R)\cap(\rmm[R]+\k(R)_\R)=\k(R)_\R$.
Hence $\overline{\c_r(R)}\cap-\overline{\c_r(R)}\subseteq\overline{\c_r(R)}\cap\k(R)_\R$.
Since we have shown (f) implies (e), we observe (f) implies (g).

Suppose that $\c_r^0(R)$ is bounded and $\h(R)$ is finitely generated.
It is seen from Proposition \ref{16} that there exist only finitely many maximal Cohen--Macaulay points in $\c_r^0(R)$, say $[M_1],\dots,[M_u]$.
As any rank $r$ maximal Cohen--Macaulay point in $\c_r(R)$ appears in $\c_r^0(R)$, the rank $r$ maximal Cohen--Macaulay points in $\c_r(R)$ are $[M_1],\dots,[M_u]$.
We have shown (f) implies (b), assuming $\h(R)$ is finitely generated.

Assume $R$ is normal and $[K_1],\dots,[K_t]$ are the rank $1$ maximal Cohen--Macaulay points in $\c_1(R)$, where each $K_i$ is a reflexive ideal of $R$.
Let $L$ be a nonzero reflexive ideal of $R$ which is a maximal Cohen--Macaulay $R$-module.
Then $[L]=[K_l]$ in $\h(R)_\R$ for some $1\le l\le t$, which implies $[R/L]=[R/K_l]$ in $\k(R)_\R$.
Sending this by $\cl_\R$, we get $L=K_l$ in $\Cl(R)_\R$.
Thus $K_1,\dots,K_t$ are the maximal Cohen--Macaulay points in $\Cl(R)_\R$.
Suppose that $\h(R)$ is finitely generated, so that $\h(R)=\Z[R]+\sum_{i=1}^\zeta\Z[\Phi_i]+\sum_{j=1}^\eta\Z[\Psi_j]$ (by Setup \ref{20}(1)).
For each $1\le k\le t$, choose integers $a_{k1},\dots,a_{k\zeta},b_{k1},\dots,b_{k\eta}$ such that $[R/K_k]=\sum_{i=1}^\zeta a_{ki}[\Phi_i]+\sum_{j=1}^\eta b_{kj}[\Psi_j]$ in $\k(R)$.
The images of the maximal Cohen--Macaulay points in $\Cl(R)$ by the isomorphism $\Cl(R)\to\Z^{\oplus\zeta}\oplus(\bigoplus_{j=1}^\eta\Z/\theta_j\Z)$ are contained in the finite set $\{-(a_{k1},\dots,a_{k\zeta},\overline{c_1}\dots,\overline{c_s})\mid1\le k\le t,\,0\le c_j\le \theta_j-1\}$.
Hence there exist only finitely many maximal Cohen--Macaulay points in $\Cl(R)$.
We have proved that (b) implies (a), assuming $R$ is normal, $\h(R)$ is finitely generated and $r=1$.

It remains to prove that (e) $\Rightarrow$ (b) under the assumption that $\h(R)$ is finitely generated.
Suppose that there are infinitely many maximal Cohen--Macaulay points $[K_1],[K_2],[K_3],\dots$ in $\c_r(R)$.
For each $k\in\Z_{>0}$ there exist $a_{ki},b_{kj}\in\Z$ such that $[K_k]=r[R]+\sum_{i=1}^\zeta a_{ki}[\Phi_i]+\sum_{j=1}^\eta b_{kj}[\Psi_j]$ in $\h(R)$.
If the $a_{ki}$ were zero, then $[K_k]$ would belong to $r[R]+\sum_{j=1}^\eta\Z[\Psi_j]$, which is a finite set since $\theta_j[\Psi_j]=0$.
Hence for infinitely many $k$ there is $i$ with $a_{ki}\ne0$.
Changing indices, we may assume that for all $k$ there is $i$ with $a_{ki}\ne0$.
We have $z_k:=[K_k]-r[R]=\sum_{i=1}^\zeta a_{ki}[\Phi_i]$ in $\h(R)_\R$.
Note that $c_k:=\|z_k\|=\sqrt{\sum_{i=1}^\zeta a_{ki}^2}>0$ for all $k$.
The point $c_k^{-1}z_k$ belongs to the unit sphere $S=\{v\in\k(R)_\R\mid\|v\|=1\}$ in $\k(R)_\R$.
Again, $S$ is a sequentially compact subset of $\k(R)_\R$, and we find a subsequence of $\{c_k^{-1}z_k\}_{k=1}^\infty$ that converges to a point in $S$.
Changing indices, we may assume $\lim_{k\to\infty}c_k^{-1}z_k=:\sigma\in S$.

Suppose that the sequence $\{c_k\}_{k=1}^\infty$ of positive real numbers is unbounded.
Then it contains a subsequence $\{c_{k_l}\}_{l=1}^\infty$ such that $\lim_{l\to\infty}c_{k_l}=\infty$.
Since the subsequence $\{c_{k_l}^{-1}z_{k_l}\}_{l=1}^\infty$ of the sequence $\{c_k^{-1}z_k\}_{k=1}^\infty$ converges to $\sigma$ as well, we have $\lim_{l\to\infty}c_{k_l}^{-1}([K_{k_l}]-r[R])=\sigma$, which implies
\begin{equation}\label{50}
\textstyle\lim_{l\to\infty}c_{k_l}^{-1}[K_{k_l}]=\sigma.
\end{equation}
Since $c_{k_l}^{-1}[K_{k_l}]$ belongs to $\c_r(R)$, the point $\sigma$ belongs to its closure $\overline{\c_r(R)}$.
We get $\sigma\in\overline{\c_r(R)}\cap\k(R)_\R=0$, which contradicts the fact that $\|\sigma\|=1$.
Therefore, the sequence $\{c_k\}_{k=1}^\infty$ is bounded.

Via the isomorphism $\k(R)_\R\cong\R^{\oplus\zeta}$, the point $z_k=\sum_{i=1}^\zeta a_{ki}[\Phi_i]\in\k(R)_\R$ corresponds to the lattice point $(a_{k1},\dots,a_{k\zeta})\in\R^{\oplus\zeta}$.
As $\{c_k\}_{k=1}^\infty$ is bounded, we see that the set $\{(a_{k1},\dots,a_{k\zeta})\}_{k=1}^\infty$ is finite, and so is the set $\{(b_{k1},\dots,b_{k\eta})\}_{k=1}^\infty$ since the subgroup $\sum_{j=1}^\eta\Z[\Psi_j]$ of $\k(R)$ is finite.
Consequently, there are only finitely many choices of the $[K_k]$ as elements of $\h(R)$.
This contradiction finishes the proof.
\end{proof}

\begin{rem}
Comparing Propositions \ref{15-0}, \ref{48}, \ref{15} and \ref{1p}, the reader may wonder if (b)--(f) in Theorem \ref{1} are equivalent when $\h(R)$ is finitely generated, even if $\c_r(R)$ is replaced with $\c(R)$.
However, we do not know whether it does.
In fact, under this replacement, we do not know how to deduce \eqref{50}.
\end{rem}

We record a consequence of Theorem \ref{1}.

\begin{cor}
\begin{enumerate}[\rm(1)]
\item
Let $R$ be a Cohen--Macaulay normal domain having a canonical module such that $\h(R)$ is finitely generated and $\dim_\R\k(R)_\R\le1$.
Suppose that there exist infinitely many nonisomorphic maximal Cohen--Macaulay $R$-modules of rank one.
Then $\overline{\c_1(R)}=\rp[R]+\k(R)_\R$.
\item
Let $E$ be an elliptic curve over $\Q$ of positive rank, and let $R$ be the local ring at the vertex of the cone over $E$.
Then $\overline{\c_1(R)}\cap\k(R)_\R$ is an $\R$-vector space of positive dimension.
\end{enumerate}
\end{cor}

\begin{proof}
(1) Proposition \ref{1p}(1) and the implication (e) $\Rightarrow$ (a) in Theorem \ref{1} infer that $\overline{\c_1(R)}\cap\k(R)_\R$ is a nonzero subspace of the $\R$-vector space $\k(R)_\R$.
Since $\dim_\R\k(R)_\R\le1$, we have $\overline{\c_1(R)}\cap\k(R)_\R=\k(R)_\R$.
Hence $\overline{\c_1(R)}$ contains $\k(R)_\R$.
By definition, $\c_1(R)$ contains $\rp[R]$.
Thus the inclusion $\overline{\c_1(R)}\supseteq\rp[R]+\k(R)_\R$ holds.
The opposite inclusion follows from Proposition \ref{15}(3).

(2) The assertion is a consequence of the combination of Proposition \ref{1p}(1), Theorem \ref{1} and \cite[Example 4.2 and the preceding paragraph]{DK}.
\end{proof}

We establish a lemma to show our next result.

\begin{lem}\label{30}
Let $R$ be a local ring of dimension at most three such that $\h(R)_\R$ has finite dimension.
Let $L$ be an $R$-module of finite length and finite projective dimension, and put $X=\{x\in\h(R)_\R\mid\chi(L,x)\ge0\}$.
Then $X$ is a closed subset of $\h(R)_\R$, and contains $\overline{\c_r(R)}$ for each integer $r>0$.
\end{lem}

\begin{proof}
The first assertion is a direct consequence of Lemma \ref{30-0}(2); recall that a linear map of finite-dimensional $\R$-vector spaces is continuous.
We prove the second assertion.
Thanks to the first assertion, it suffices to check that $X$ contains $\c_r(R)$.
Let $z\in\c_r(R)$.
Then $z=\sum_{i=1}^na_i[M_i]$ for some $a_i\in\rp$ and $M_i\in\cm(R)$ with $\rk M_i=r$.
It holds that $\chi(L,z)=\sum_{i=1}^na_i\cdot\chi(L,M_i)$.
As $\depth M_i=\dim R$, we see from \cite[Page 140, Lemma 2(iii)]{M} that $\Tor_{>0}^R(L,M_i)=0$.
Hence $\chi(L,M_i)=\ell_R(L\otimes_RM_i)\ge0$ for each $i$, which shows $\chi(L,z)\ge0$.
Therefore $z$ belongs to $X$ as desired.
\end{proof}

As an application of Theorem \ref{1}, we obtain the following result.

\begin{cor}\label{25}
Let $R$ be a homomorphic image of a Gorenstein local ring, and suppose that $R$ is a normal ring with $\dim R\le3$ such that $\Cl(R)$ is finitely generated and of rank $1$.
Assume that there exist an $R$-module $L$ of finite length and finite projective dimension and a torsion $R$-module $T$ with $\chi(L,T)\ne0$.
Then there exist only finitely many maximal Cohen--Macaulay $R$-modules of rank $1$ up to isomorphism.
\end{cor}

\begin{proof}
As $\chi(L,T)\ne0$, we have $L\ne0$.
Thanks to the (New) Intersection Theorem, $R$ is Cohen--Macaulay (see \cite[Exercise 9.4.9]{BH}).
As $R$ is a quotient of a Gorenstein ring, it has a canonical module.
By Theorem \ref{1}, it suffices to show $\overline{\c_1(R)}\cap\k(R)_\R=0$.
As $\chi(L,-T)=-\chi(L,T)$, there exists $\alpha\in\k(R)_\R$ with $\chi(L,\alpha)<0$.
Then $\alpha$ is not in the set $X$ given in Lemma \ref{30}.
The lemma shows $\alpha\notin\overline{\c_1(R)}$.
Proposition \ref{1p} says $\overline{\c_1(R)}\cap\k(R)_\R$ is a subspace of $\k(R)_\R\cong\Cl(R)_\R\cong\R$, and we must have $\overline{\c_1(R)}\cap\k(R)_\R=0$.
\end{proof}

\begin{rem}\label{12}
\begin{enumerate}[(1)]
\item
The assumption of Corollary \ref{25} on the intersection multiplicity is equivalent to the existence of a torsion $R$-module which is not numerically trivial by Lemma \ref{30-0}(1) (note that $R$ must be Cohen--Macaulay as we saw at the beginning of the proof of Corollary \ref{25}).
\item
Corollary \ref{25} can also be deduced by using a result of Dao and Kurano.
Since $T$ is a torsion $R$-module which is not numerically trivial, we must have $\dim R=3$ (see \cite[Proposition 2.2 and the preceding part]{CK}).
Let $\rho:\A_2(R)\to\overline{\A_2(R)}$ be the natural map, where $\A_\ast(R)$ and $\overline{\A_\ast(R)}$ denote the Chow group of $R$ and its quotient by numerical equivalence, respectively.
Since $\Cl(R)$ is finitely generated and of rank one, $\A_2(R)$ is finitely generated as well and $\dim_\Q\A_2(R)_\Q=1$.
As we saw in the proof of Corollary \ref{25}, we may assume $T=R/\p$ for some prime ideal $\p$ of $R$ with height one.
This gives an element of $\A_2(R)_\Q$ which does not belong to $\ker(\rho_\Q)$, and hence $\rho_\Q\ne0$.
We see that $(\ker\rho)_\Q=0$, which implies that $\ker\rho$ is finite.
It follows from \cite[Corollary 4.6]{DK} that there exist only finitely many maximal Cohen--Macaulay $R$-modules of rank one up to isomorphism.
(Since $R$ is an integrally closed domain of dimension three, finite generation of $\Cl(R)$ implies finite generation of $\overline{\g(R)}$.
The condition \cite[Assumption 2.1]{DK} in \cite[Corollary 4.6]{DK} is used to get finite generation of $\overline{\g(R)}$.)
\end{enumerate}
\end{rem}

\begin{ex}\label{27}
Let $k$ be a field.
Consider $R=k[\![x,y,z,w]\!]/(xw-yz)$ and $\p=(x,y)R$.
Then $R$ is a $3$-dimensional local hypersurface with an isolated singularity, $\p$ is a prime ideal of $R$ with height one, and $\Cl(R)=\Z\p\cong\Z$ by \cite[Proposition 14.8]{F}.
By virtue of \cite{DHM}, there exists an $R$-module $L$ of finite length and finite projective dimension such that $\chi(L,R/\p)=-1\ne0$.
We can apply Corollary \ref{25} to deduce that there exist only finitely many nonisomorphic maximal Cohen--Macaulay $R$-modules of rank one\footnote{This statement is itself an immediate consequence of the fact that $R$ has finite Cohen--Macaulay representation type by \cite[Corollary (12.6)]{Y} (at least when $k$ is an algebraically closed field of characteristic zero).}.
\end{ex}

\section{Maximal Cohen--Macaulay points in $\Cl(R)$}

In the previous section we give several equivalent conditions for a normal domain to possess only finitely many maximal Cohen--Macaulay modules of rank one.
The main purpose of this section is to try to approach more directly the problem asking when this is the case.

We begin with the lemma below, which is a generalization of \cite[Proposition 2.5.1]{I}.

\begin{lem}\label{4}
Let $R$ be a local ring.
Let $M,N$ be $R$-modules and $n\ge2$ an integer.
Suppose that $M$ is locally free on the punctured spectrum of $R$, that $\depth_R N\ge n-1$ and that $\depth_R\Hom_R(M,N)\ge n$.
It then holds that $\Ext_R^i(M,N)=0$ for all $1\le i\le n-2$.
\end{lem}

\begin{proof}
Assume that the conclusion does not hold.
Putting $l=\inf\{i\ge1\mid\Ext^i(M,N)\ne0\}$, we have $1\le l\le n-2$.
Let $F$ be a minimal free resolution of $M$.
Dualizing this by $N$, we get an exact sequence
$$
0\to\Hom(M,N)\to\Hom(F_0,N)\to\cdots\to\Hom(F_{l-1},N)\to\Hom(\syz^lM,N)\to\Ext^l(M,N)\to0.
$$
Since $M$ is locally free on the punctured spectrum of $R$, the $R$-module $\Ext^l(M,N)$ has finite length.
As $\depth N\ge n-1\ge1$, we have $\depth\Hom(\syz^lM,N)\ge1$ and $\depth\Hom(F_j,N)\ge n-1$ for all $0\le j\le l-1$.
The depth lemma shows $\depth\Hom(M,N)=l+1$, which contradicts the assumption $\depth\Hom(M,N)\ge n$.
This completes the proof of the lemma.
\end{proof}

For a module $M$ over a local ring $R$, we denote by $\nu_R(M)$ the minimal number of generators of $M$.
Recall that $M$ is called {\em rigid} if $\Ext_R^1(M,M)=0$.
We record three cases where there are only finitely many nonisomorphic maximal Cohen--Macaulay modules of rank one.

\begin{prop}\label{18}
There exist only finitely many nonisomorphic maximal Cohen--Macaulay $R$-modules of rank one in each of the following three cases.
\begin{enumerate}[\rm(1)]
\item
$R$ is a $3$-dimensional local hypersurface with an isolated singularity, which has a desingularization $f:X\to\spec R$ such that $X\setminus f^{-1}(\m)\simeq\spec R\setminus\{\m\}$, where $\m$ is the maximal ideal of $R$.
\item
$R$ is a $d$-dimensional Cohen--Macaulay local ring with $d\ge3$ and with an isolated singularity, which is the completion of a finitely generated graded algebra $A=\bigoplus_{i\ge0}A_i$ over an algebraically closed field $k$ with $\dim_kA_0<\infty$.
\item
$R$ is a $d$-dimensional complete local hypersurface with $d\ge2$ and with algebraically closed uncountable coefficient field of characteristic not $2$, which has countable Cohen--Macaulay representation type.
\end{enumerate}
\end{prop}

\begin{proof}
(1) The assertion is shown in \cite[Theorem 4.7 and Corollary 4.8]{DK}.

(2) Let $M$ be a maximal Cohen--Macaulay $R$-module of rank $1$.
Then $M$ corresponds to an element of $\Cl(R)$.
We have $\Hom_R(M,M)=M-M=0$ in $\Cl(R)$, which means $\Hom_R(M,M)\cong R$.
Lemma \ref{4} implies that $M$ is rigid, as $d\ge3$.
The assertion follows from \cite[Corollary B]{DS}.

(3) We may assume that $R$ has infinite Cohen--Macaulay representation type.
Then $R$ is a hypersurface of type $(\A_\infty)$ or $(\D_\infty)$, and the isomorphism classes of indecomposable maximal Cohen--Macaulay $R$-modules are classified completely \cite{BGS,BD}.
Fix a nonfree maximal Cohen--Macaulay $R$-module $M$ of rank one.
If $d=2$, then it is directly seen from \cite[Theorems 5.3 and 5.7]{BD} that $\nu_R(M)\le2$ and there are only finitely many choices of $M$.
Let $d\ge3$.
Then $R$ is a (normal) domain, and $\nu_R(M)\le\e(M)=\e(R)\cdot\rk_RM=\e(R)=2$.
By Kn\"orrer's periodicity (\cite[Chapter 12]{Y}), we can write $R=T[\![x,y]\!]/(f+xy)$, where $T$ is a formal power series ring in $d-1$ variables, $f\in T$ and $S=T/(f)$ is a $(d-2)$-dimensional hypersurface of the same type as $R$.
Via the equivalence $\underline\cm(S)\to\underline\cm(R)$ of stable categories of maximal Cohen--Macaulay modules, the nonfree maximal Cohen--Macaulay $R$-module $M$ comes from a nonfree maximal Cohen--Macaulay $S$-module $N$ such that $\nu_R(M)=2\nu_S(N)$.
The fact that $\nu_R(M)\le2$ forces us to have $\nu_R(M)=2$ and $\nu_S(N)=1$.
If $\dim S$ were at least three, then $S$ would be a (normal) domain.
A domain cannot possess a nonfree cyclic maximal Cohen--Macaulay module.
Hence $\dim S\le2$.
It is observed from \cite[Theorems 5.3 and 5.7]{BD} and \cite[Propositions 4.1 and 4.2]{BGS} that there exist only a finite number of isomorphism classes of cyclic maximal Cohen--Macaulay $S$-modules.
Thus there are only finitely many choices of $N$, and the same thing holds for $M$.
\end{proof}

Let $R$ be a local ring.
Let $n>0$ be an integer.
An $R$-module $M$ is called {\em $n$-periodic} if $\syz_R^nM\cong M$.
By definition, a nonzero free $R$-module cannot be $n$-periodic for all $n>0$.
If $R$ is Cohen--Macaulay, then an $n$-periodic $R$-module is maximal Cohen--Macaulay, since it is an $m$th syzygy for some $m\ge\dim R$.
To get our next result, we state a lemma.

\begin{lem}\label{7}
Let $n$ be a positive integer.
Let $R$ be a local ring with $\depth R\ge n+1$.
Let $M\ne0$ be an $n$-periodic $R$-module with $\Ext_R^i(M,R)=0$ for all $1\le i\le n$, and assume that $M$ is locally free on the punctured spectrum of $R$.
One then has $\depth_R\Hom_R(M,M)\le n+1$.
\end{lem}

\begin{proof}
First of all, we observe that there is an exact sequence
$$
\cdots\xrightarrow{\partial_2}F_1\xrightarrow{\partial_1}F_0\xrightarrow{f}F_{n-1}\xrightarrow{\partial_{n-1}}F_{n-2}\xrightarrow{\partial_{n-2}}\cdots\xrightarrow{\partial_2}F_1\xrightarrow{\partial_1}F_0\xrightarrow{f}F_{n-1}\xrightarrow{\partial_{n-1}}F_{n-2}\xrightarrow{\partial_{n-2}}\cdots
$$
of free $R$-modules such that $\im f=M=\syz^nM\subseteq\m F_{n-1}$ and $\im\partial_i\subseteq\m F_{i-1}$ for $1\le i\le n-1$.
Since $\Ext^i(M,R)=0$ for $1\le i\le n$, we see that the $R$-dual of the above sequence is exact as well.
Hence $M$ is totally reflexive, and $\depth M=\depth R\ge n+1$; see \cite[Theorems (1.4.8) and (4.2.6)]{C}.

Suppose that $\depth\Hom(M,M)\ge n+2$.
Then Lemma \ref{4} implies $\Ext^i(M,M)=0$ for all $1\le i\le n$.
As $\syz^nM\cong M$, we have $0=\Ext^n(M,M)\cong\lhom(\syz^nM,M)\cong\lhom(M,M)$, which implies that the identity map of $M$ factors through some free $R$-module $P$; we refer the reader to \cite[\S7]{catgp} for basic properties of stable hom-sets.
Therefore $M$ is a direct summand of $P$, and thus $M$ is free.
This contradicts the assumption that $M$ is $n$-periodic, and it follows that $\depth\Hom(M,M)\le n+1$.
\end{proof}

Now we can prove the following proposition concerning $n$-periodic ideals of $R$, that is, ideals of $R$ that are $n$-periodic as $R$-modules.

\begin{prop}\label{6}
\begin{enumerate}[\rm(1)]
\item
Let $R$ be a normal local ring of depth $t$.
Let $1\le n\le t-2$ be an integer.
Let $I\ne0$ be an $n$-periodic ideal of $R$.
Suppose that $I$ is locally free on the punctured spectrum of $R$.
Then $\Ext_R^i(I,R)\ne0$ for some $1\le i\le n$.
\item
Let $R$ be an Gorenstein local ring with an isolated singularity.
Let $1\le n\le\dim R-2$ be an integer.
Then there exists no nonzero $n$-periodic ideal of $R$.
\item
Let $R$ be a local hypersurface of dimension at least four with an isolated singularity.
Then any maximal Cohen--Macaulay $R$-module of rank one is free.
\end{enumerate}
\end{prop}

\begin{proof}
(1) Since $R$ is normal and $I\ne0$, we have $\Hom_R(I,I)\cong R$.
The assertion follows by Lemma \ref{7}.

(2) Assume that there exists a nonzero $n$-periodic ideal $I$ of $R$.
Then $I$ is a maximal Cohen--Macaulay $R$-module, and $\Ext_R^i(I,R)=0$ for all $i>0$ as $R$ is Gorenstein.
Since $R$ has an isolated singularity, $I$ is locally free on the punctured spectrum of $R$.
In view of (1), we have a contradiction.

(3) Suppose that there is a nonfree maximal Cohen--Macaulay $R$-module $M$ of rank $1$.
Then $M$ is isomorphic to a nonzero ideal, and $2$-periodic as $R$ is a hypersurface.
By (2) we get a contradiction.
\end{proof}

\begin{rem}
\begin{enumerate}[(1)]
\item
A more general statement than Proposition \ref{6}(3) is known, namely, a local complete intersection which is factorial in codimension three is factorial; see \cite[Corollary]{CL}.
\item
The assumption that $\dim R\ge4$ in Proposition \ref{6}(3) is indispensable.
In fact, let $R,\p$ be as in Example \ref{27}.
Then $\p$ is a nonfree maximal Cohen--Macaulay $R$-module of rank one.
\end{enumerate}
\end{rem}

Let $I$ be an ideal of $R$.
We denote by $\cd I$ the {\em cohomological dimension} of $I$, that is to say, $\cd I=\sup\{i\in\Z\mid\H^i_I(R)\ne0\}$.
Also, recall that a local ring $R$ is called {\em analytically irreducible} if the completion $\widehat R$ of $R$ is a domain.
To prove the next theorem, we establish a lemma.

\begin{lem}\label{26}
Let $R$ be a $d$-dimensional Cohen--Macaulay local ring.
Let $I$ be an ideal of $R$ which is locally free on the punctured spectrum of $R$.
Then the following statements hold true.
\begin{enumerate}[\rm(1)]
\item
For each positive integer $n$, the power $I^n$ is locally free on the punctured spectrum of $R$.
\item
If $I^\ast$ is a maximal Cohen--Macaulay $R$-module, then $\Ext_R^i(R/I,R)=0$ for all $2\le i\le d-1$.
\item
Suppose that $R$ is analytically irreducible with $d\ge2$ and that $I$ is a nonzero reflexive ideal.
Let $\{I^{a_i}\}_{i=1}^\infty$ be a cofinal subfiltration of the $I$-adic filtration $\{I^n\}_{n=1}^\infty$.
If the $R$-module $(I^{a_i})^\ast$ is maximal Cohen--Macaulay for all $i>0$, then $\cd I=1$.
\end{enumerate}
\end{lem}

\begin{proof}
(1) Let $\p$ be a nonmaximal prime ideal of $R$.
By assumption, $IR_\p$ is $R_\p$-free, and $IR_\p\cong R_\p^{\oplus k}$ for some $k\ge0$.
Hence $I^2R_\p=I(IR_\p)\cong IR_\p^{\oplus k}\cong R_\p^{\oplus k^2}$.
Iterating this procedure shows that $I^nR_\p$ is $R_\p$-free.

(2) Lemma \ref{4} implies $\Ext_R^i(I,R)=0$ for all $1\le i\le d-2$.
Thus $\Ext_R^i(R/I,R)=0$ for all $2\le i\le d-1$.

(3) As $R$ is a domain and $I\ne0$, we have $\H_I^0(R)=0$.
Since $I$ is reflexive, it has depth at least two as an $R$-module by \cite[Exercise 1.4.19]{BH}.
The depth lemma says that the local ring $R/I$ has positive depth, and in particular, it is non-artinian.
As the completion of $R$ is a domain, by virtue of the Hartshorne--Lichtenbaum vanishing theorem \cite[8.2.1]{BS} we obtain $\H_I^j(R)=0$ for all $j\ge d$.
Fix an integer $i>0$.
It follows from (1) and (2) that $\Ext_R^j(R/I^{a_i},R)=0$ for all $2\le j\le d-1$.
As the subfiltration $\{I^{a_i}\}_{i=1}^\infty$ of the $I$-adic filtration $\{I^n\}_{n=1}^\infty$ is cofinal, we get $\H_I^j(R)\cong\varinjlim_i\Ext_R^j(R/I^{a_i},R)=0$ for all $2\le j\le d-1$; see \cite[Exercise 5.22]{R}.
Consequently, we have $\H_I^j(R)=0$ for all $j\ne1$, and we conclude that $\cd I=1$.
\end{proof}

We state an elementary fact on divisor class groups which is used in the proof of our theorem.

\begin{rem}\label{22}
Let $R$ be a normal domain.
Let $I$ be a nonzero reflexive ideal of $R$.
Let $n$ be a positive integer.
Then $(I^n)^\ast$ is a nonzero reflexive ideal of $R$ as well, and $(I^n)^\ast=-nI$ in $\Cl(R)$.
Indeed, the kernel of the natural surjection $I^{\otimes n}\to I^n$ is torsion by a rank computation.
This yields $(I^n)^\ast\cong(I^{\otimes n})^\ast$.
Let $C$ be the cokernel of the natural injection $I^n\to(I^n)^{\ast\ast}$.
Let $\p\in\spec R$ with $\depth R_\p\le1$.
As $R$ is normal, we have $\height\p\le1$, and $R_\p$ is regular.
Hence $IR_\p\cong R_\p$, which implies $C_\p=0$.
Therefore $\grade C\ge2$, and we get $(I^n)^{\ast\ast\ast}\cong(I^n)^\ast$.
Thus $(I^n)^\ast\cong(I^{\otimes n})^{\ast\ast\ast}$.
It remains to note that $(I^{\otimes n})^{\ast\ast\ast}=-nI$ in $\Cl(R)$.
\end{rem}

Now we can prove the theorem below, which is the first main result in this section.
Note that the limit $\lim_{n\to\infty}\depth R/I^n$ always exists; see \cite[Theorem (2)]{B}.

\begin{thm}\label{8}
Let $(R,\m)$ be a Gorenstein local ring of dimension $d\ge2$ with an isolated singularity.
Let $I$ be a nonzero reflexive ideal of $R$.
Assume that either of the following conditions is satisfied.
\\
\qquad
{\rm(1)} $R$ is analytically irreducible and $\cd I\ne1$.\qquad
{\rm(2)} $\lim_{n\to\infty}\depth R/I^n\ne0$.
\\
Then there exist only a finite number of maximal Cohen--Macaulay points on $\Z I$ in $\Cl(R)$.
In case {\rm(1)}, one further has that there is an isomorphism $\Z I\cong\Z$ of $\Z$-modules.
\end{thm}

\begin{proof}
Let $\p$ be a nonmaximal prime ideal of $R$.
If $I\nsubseteq\p$, then $IR_\p=R_\p$.
Assume $I\subseteq\p$.
The ideal $IR_\p$ of $R_\p$ is reflexive, and belongs to $\Cl(R_\p)$.
As $R$ is a (normal) domain, $IR_\p$ is seen to be nonzero.
The local ring $R_\p$ is regular as $R$ has an isolated singularity, and $\Cl(R_\p)$ is trivial.
Hence $IR_\p\cong R_\p$.
It follows that the $R$-module $I$ is locally free on the punctured spectrum of $R$.

Let us show the last assertion of the theorem.
Assume that $I$ is torsion in $\Cl(R)$.
Then there is an integer $t>0$ with $tI=0$ in $\Cl(R)$.
The subfiltration $\{I^{it}\}_{i=1}^\infty$ of the $I$-adic filtration $\{I^n\}_{n=1}^\infty$ is cofinal.
For each $i>0$, we have $(I^{it})^\ast=-itI=0$ in $\Cl(R)$ by Remark \ref{22}, which means $(I^{it})^\ast\cong R$.
Thus $\cd I=1$ by Lemma \ref{26}(3).
This contradiction shows that $I$ is not torsion in $\Cl(R)$, which means $\Z I\cong\Z$.

Now, we prove the first assertion of the theorem.
Suppose that there exist infinitely many maximal Cohen--Macaulay points on $\Z I$.
Then we find pairwise distinct integers $a_1,a_2,a_3,\dots$ such that the points $a_iI\in\Cl(R)$ are maximal Cohen--Macaulay.
Take a nonzero reflexive ideal $I_i$ of $R$ which is a maximal Cohen--Macaulay $R$-module and satisfies $I_i=a_iI$ in $\Cl(R)$.
We have $I_i^\ast=-a_iI$ in $\Cl(R)$, and $I_i^\ast$ is also a maximal Cohen--Macaulay $R$-module since $R$ is Gorenstein.
Hence we may assume that $0<a_1<a_2<a_3<\cdots$.
For each $i>0$ there are equalities $I_i^\ast=-a_iI=(I^{a_i})^\ast$ in $\Cl(R)$ by Remark \ref{22}, which shows that the $R$-module $(I^{a_i})^\ast$ is maximal Cohen--Macaulay.
It follows from assertions (1) and (2) of Lemma \ref{26} that $\Ext_R^j(R/I^{a_i},R)=0$ for all integers $i\ge1$ and $2\le j\le d-1$.

(1) As $0<a_1<a_2<a_3<\cdots$, the subfiltration $\{I^{a_i}\}_{i=1}^\infty$ of the $I$-adic filtration $\{I^n\}_{n=1}^\infty$ is cofinal.
By assumption, $R$ is analytically irreducible and $\cd I\ne1$.
Lemma \ref{26}(3) gives rise to a contradiction.
We conclude that there exist only finitely many maximal Cohen--Macaulay points on $\Z I$.

(2) Our assumption that there are infinitely many maximal Cohen--Macaulay points on $\Z I$ forces us to have $I\ne R$.
Since $\lim_{n\to\infty}\depth R/I^n>0$, we have $\depth R/I^n>0$ for all $n\gg0$.
There are infinitely many integers $n>0$ with $\Ext_R^j(R/I^n,R)=0$ for all $2\le j\le d-1$ and $\depth R/I^n>0$.
As $R$ is a domain, we have $\dim R/I^n<d$.
It follows from \cite[Corollary 3.5.9]{BH} that $\H_\m^k(R/I^n)=0$ for all $k\le d-2$.
Thus $R/I^n$ is a Cohen--Macaulay local ring of dimension $d-1$, and we see that $\height I=1$.
Take any minimal prime $\p$ of $I$.
Then $\p/I^n$ is a minimal prime of $R/I^n$.
We have $\dim R/\p=\dim(R/I^n)/(\p/I^n)=\dim R/I^n=d-1$, and hence $\height\p=1$.
As $R$ is normal, $R_\p$ is a discrete valuation ring, and $IR_\p$ is a principal ideal of $R_\p$.
We can apply \cite[Corollary (14)]{B} to deduce that $I$ is a principal ideal of $R$.
Therefore $I=0$ in $\Cl(R)$, which contradicts our assumption that there are infinitely many maximal Cohen--Macaulay points on $\Z I$.
\end{proof}

\begin{rem}
In addition to the assumption of Theorem \ref{8}, suppose that the $\Z$-module $\Cl(R)$ is free.
Then there exist only finitely many maximal Cohen--Macaulay points on $\Z I$ in $\Cl(R)_\R$.
This is an easy consequence of the fact that the scalar extension $\Cl(R)\to\Cl(R)_\R$ is an injective map.
\end{rem}

The proposition below gives sufficient conditions for Theorem \ref{8}(2) to be satisfied.

\begin{prop}\label{10}
Let $(R,\m)$ be a Cohen--Macaulay local ring.
Let $I\ne\m$ be an ideal of $R$ with positive height.
Let $G=\bigoplus_{n\ge0}I^n/I^{n+1}$ be the associated graded ring of $I$.
The following implications hold.
$$
\textstyle
G\text{ is a domain}\implies
\grade\m G>0\iff
\depth R/I^n>0\text{ for all }n>0\implies
\lim_{n\to\infty}\depth R/I^n>0.
$$
\end{prop}

\begin{proof}
We have $\m G=\m/I\oplus\m I/I^2\oplus\m I^2/I^3\oplus\cdots$.
As $I\ne\m$, the ideal $\m G$ of the ring $G$ is nonzero.
The first implication immediately follows from this.
The equivalence is a consequence of \cite[Proposition (9.23)]{BV}.
The last implication is evident.
\end{proof}

\begin{ex}\label{31}
Let $k$ be a field.
\begin{enumerate}[(1)]
\item
Let $R,\p$ be as in Example \ref{27}.
Recall that $\Cl(R)=\Z\p\cong\Z$.

(a) Letting $G$ be the associated graded ring of $\p$, we have $G=(R/\p)[\p/\p^2]\cong k[\![Z,W]\!][X,Y]/(XW-YZ)$, which is a domain.
Proposition \ref{10} implies $\lim_{n\to\infty}\depth R/\p^n\ne0$.

(b) Put $J=(z,w)R$ and $\m=(x,y,z,w)R$.
Then $\p+J=\m$.
The exact sequence $0\to J\to R\to R/J\to0$ induces an exact sequence $\H_\p^2(R)\to\H_\p^2(R/J)\to\H_\p^3(J)$.
We have $\H_\p^2(R/J)\cong\H_\m^2(R/J)\ne0$ as $\dim R/J=2$.
Since $\p$ is $2$-generated, $\H_\p^3(J)=0$.
It is seen that $\H_\p^2(R)\ne0$, and therefore $\cd \p\ne1$.

Thus $R$ satisfies both (1) and (2) of Theorem \ref{8}.
It follows that there are only finitely many maximal Cohen--Macaulay points on $\Z \p=\Cl(R)$.
\item
Consider the determinantal ring $A=k[X]/\I_2(X)$, where $X=(x_{ij})$ is a $3\times3$ generic matrix.
Set $P=(x_{11},x_{12},x_{13})A$.
Let $R=k[\![X]\!]/\I_2(X)$ be the completion of $A$, and put $\p=PR$.
Using \cite[Theorems 7.3.1(c), 7.3.5, 7.3.6(b) and Proposition 7.3.4]{BH} and \cite[Remarks (8.5)]{BV}, we observe that $R$ is a $5$-dimensional Gorenstein complete local ring with an isolated singularity which is not a complete intersection, $\p$ is a prime ideal of height $1$ and $\Cl(R)\cong\Cl(A)=\Z P\cong\Z$.

(a) By \cite[Exercise 7.3.9]{BH} it holds for all $n>0$ that $P^{(n)}=P^n$, and hence $\ass A/P^n=\{P\}$.
It is easy to see that $\depth R/\p^n>0$ for all $n>0$.
Therefore $\lim_{n\to\infty}\depth R/\p^n\ne0$.

(b) Put $J=(x_{21},x_{22},x_{23},x_{31},x_{32},x_{33})R$.
The ideal $\p+J$ coincides with the maximal ideal $\m$.
An exact sequence $\H_\p^3(R)\to\H_\p^3(R/J)\to\H_\p^4(J)$ is induced.
As $\p$ is $3$-generated, it holds that $\H_\p^4(J)=0$.
Since $\dim R/J=3$, we have $\H_\p^3(R/J)=\H_\m^3(R/J)\ne0$.
Thus $\H_\p^3(R)\ne0$, and in particular, $\cd\p\ne1$.

Consequently, the ring $R$ satisfies both (1) and (2) of Theorem \ref{8}, and there exist only finitely many maximal Cohen--Macaulay points on $\Z\p=\Cl(R)$.
Note that, since it is not a hypersurface, $R$ has infinite Cohen--Macaulay representation type by \cite[Theorem (8.15)]{Y}.
\item
Let $S=k[x,y,z,w,v]/(xy+z^2+w^2+v^2,xv+yv+zw)$ be a homogeneous $k$-algebra, and let $\nn=(x,y,z,w,v)$ be the graded maximal ideal of $S$.
Assume $\ch k=0$.
Denoting by $\jac S$ the Jacobian ideal of $S$, we observe $S/\jac S$ is an artinian ring.
The Jacobian criterion shows that $S$ has an isolated singularity at $\nn$.
Let $R=k[\![x,y,z,w,v]\!]/(xy+z^2+w^2+v^2,xv+yv+zw)$ be the completion of $S$.
Then $R$ is a $3$-dimensional complete intersection with an isolated singularity.
As the a-invariant of $S$ is $-1$, it is rational by the Flenner--Watanabe theorem \cite{Fl,W}.
Hence $R$ has a rational singularity.
Let $\p=(w,v)$ be an ideal of $R$.
Then $\p$ is a prime ideal of height $1$, so $\p\in\Cl(R)$.

(a) By \cite[Theorem (1.1)]{GS} the associate graded ring $G$ of $\p$ is Gorenstein, and in particular, it is Cohen--Macaulay.
Hence $\grade\m G=\dim G-\dim G/\m G=3-\lambda(\p)$, where $\m$ is the maximal ideal of $R$.
It is seen that $w,v$ is a d-sequence.
By \cite[Corollary 5.5.5 and Exercises 8.22]{HS}, we get $\lambda(\p)=\nu(\p)=2$.
It follows that $\grade\m G=1>0$, and $\lim_{n\to\infty}\depth R/\p^n\ne0$ by Proposition \ref{10}.

(b) Set $J=(x+y,z)$.
We have $\sqrt{\p+J}=\m$, $\dim R/J=2$ and there is an exact sequence $\H_\p^2(R)\to\H_\p^2(R/J)\to\H_\p^3(J)$.
A similar argument as in (1b) shows $\H_\p^2(R)\ne0$, whence $\cd\p\ne1$.

Thus, the local ring $R$ satisfies both (1) and (2) of Theorem \ref{8}, and it follows that there exist only finitely many maximal Cohen--Macaulay points on $\Z\p\cong\Z$.
\item
Let $S=k[x,y,z,w,v]/(x^3+y^3+zwv,xy+z^2+w^2+v^2)$ be a homogeneous $k$-algebra, and assume $\ch k=0$.
Similarly to (3), $S$ is not rational (as it has a-invaiant $0$) but has an isolated singularity.
Let $R=k[\![x,y,z,w,v]\!]/(x^3+y^3+zwv,xy+z^2+w^2+v^2)$ be the completion.
Then $R$ is a $3$-dimensional complete intersection with an isolated singularity.
Set $\p=(x+y,z)R$.
This is a prime ideal of $R$ with height $1$, and hence $\p\in\Cl(R)$.
Let $\m=(x,y,z,w,v)R$ be the maximal ideal of $R$.

(a) An argument analogous to (3a) shows $\grade\m G=1>0$, where $G$ is the associate graded ring of $\p$.
Proposition \ref{10} implies that $\lim_{n\to\infty}\depth R/\p^n\ne0$.

(b) Put $J=(x^2-xy+y^2,v)R$.
There is an exact sequence $\H_\p^2(R)\to\H_\p^2(R/J)\to\H_\p^3(J)$.
An analogous argument as in (1b) shows $\H_\p^2(R)\ne0$, and hence $\cd\p\ne1$.

Thus $R$ satisfies both (1) and (2) of Theorem \ref{8}, and therefore there exist only finitely many maximal Cohen--Macaulay points on $\Z \p\cong\Z$.
\item
Let $A=k[x_1,x_2,x_3,x_4,x_5]/(x_1^3+x_2^3+x_3^3+x_4^3+x_5^3)$ and $B=k[y_1,y_2]$ be homogeneous $k$-algebras.
Let $C=A\#B$ be the Segre product and $R$ its completion with respect to the irrelevant maximal ideal.
Then $A,B$ have a-invariant $-2$.
It follows from \cite[Theorems (4.2.3), (4.4.4) and (4.4.7)]{GW} that $C$ is a $5$-dimensional Gorenstein graded ring which is not a complete intersection.
Hence $R$ is a $5$-dimensional Gorenstein complete local ring which is not a complete intersection.

We have $\Proj C\cong\Proj A\times\Proj B$, and $\Proj C$ is regular as so are $\Proj A$ and $\Proj B$.
Hence the local ring $R$ has an isolated singularity.
Let $P=(x_1y_1,x_1y_2)C$ and $\p=PR$.
Then $C/P=(A/x_1A)\#B$ is a $4$-dimensional Cohen--Macaulay ring by \cite[Theorem (4.2.3) and (4.4.4)]{GW}, and $\Proj(C/P)\cong\Proj(A/x_1A)\times\Proj B$ is nonsingular.
Hence $R/\p$ is a $4$-dimensional Cohen--Macaulay local ring with an isolated singularity, and in particular, it is a domain.
Thus $\p$ is a prime ideal of $R$ with height one, and hence $\p\in\Cl(R)$.

(a) Let $\m$ be the maximal ideal of $R$.
A similar argument as in (3a) shows that $\grade\m G=\dim G-\lambda(\p)=5-2=3$, where $G$ is the associated graded ring of $\p$, and we get $\lim_{n\to\infty}\depth R/\p^n\ne0$.

(b) Set $J=(x_3,x_4,x_5)A$, $K=J\#B$ and $D=(A/J)\#B$.
The exact sequence $0\to K\to C\to D\to0$ of $C$-modules (see \cite[Remark (4.0.3)]{GW}) induces an exact sequence $\H_P^2(C)\to\H_P^2(D)\to\H_P^3(K)$.
We have $\H_P^3(K)=0$ as $\nu(P)=2$.
It is easy to see that $PD$ is a primary ideal of $D$ whose radical coincides with the irrelevant maximal ideal of $D$.
As $\dim D=2$ by \cite[Theorem (4.2.3)]{GW}, we have $\H_P^2(D)\ne0$.
The above exact sequence shows $\H_P^2(C)\ne0$, which implies $\H_\p^2(R)\ne0$.
Thus $\cd\p\ne1$.

Consequently, $R$ satisfies both (1) and (2) of Theorem \ref{8}, and it follows that there exist only finitely many maximal Cohen--Macaulay points on $\Z\p\cong\Z$.
\end{enumerate}
\end{ex}

\begin{rem}
Bruns and Gubeladze \cite[Corollary 4.3.2]{BG} prove that there are only finitely many maximal Cohen--Macaulay points in the divisor class group of a normal semigroup ring.
The reader may wonder if this applies to get the conclusion of (3) and (4) of Example \ref{31}, but it does not.
Indeed, in each statement, the defining ideal of the ring $S$ is not a binomial ideal, and hence $S$ is not toric, so that one cannot apply \cite[Corollary 4.3.2]{BG} to the ring $S$.
\end{rem}

For each homomorphism $h:X\to R$ of $R$-modules, we denote by $h'$ the homomorphism $R\to X^\ast$ of $R$-modules given by $1\mapsto h$.
To show our next result, we establish a lemma, which may be well-known.

\begin{lem}\label{60}
Let $I$ be a reflexive ideal of $R$ with positive grade.
Suppose that $I^\ast$ is generated by two elements $f$ and $g$.
Then the sequence $0\to I\xrightarrow{\binom{-g}{f}}R^{\oplus2}\xrightarrow{(f',g')}I^\ast\to0$ is exact.
\end{lem}

\begin{proof}
Clearly, $(f',g')$ is surjective.
For $x,y\in I$ we have $(-g(x)f+f(x)g)(y)=-g(x)f(y)+f(x)g(y)=-f(g(x)y)+f(g(y)x)=-f(g(xy))+f(g(xy))=0$, which shows $(f',g')\binom{-g}{f}=0$.
The map $\phi:I\to I^{\ast\ast}$ given by $\phi(x)(h)=h(x)$ for $x\in I$ and $h\in I^\ast$ is an isomorphism.
If $x\in I$ is such that $f(x)=g(x)=0$, then $h(x)=0$ for all $h\in I^\ast$, which means $\phi(x)=0$, and hence $x=0$.
This shows the injectivity of $\binom{-g}{f}$.

Let $a,b\in R$ be such that $af+bg=0$.
Define $\psi:I^\ast\to R$ by $\psi(pf+qg)=aq-bp$ for $p,q\in R$.
Assume that $pf+qg=rf+sg$ for $p,q,r,s\in R$.
Then $A\binom{f}{g}=\binom{0}{0}$, where $A:=\left(\begin{smallmatrix}a&b\\p-r&q-s\end{smallmatrix}\right)$.
Multiplying it by the adjugate of the matrix $A$ gives $t\binom{f}{g}=\binom{0}{0}$, where $t:=a(q-s)-b(p-r)$.
Hence $tI^\ast=0$, and $tI=0$ as $I$ is reflexive.
Since $I$ has positive grade, we get $t=0$ and $aq-bp=as-br$.
This shows that $\psi$ is a well-defined map.
As $\psi$ belongs to $I^{\ast\ast}$, we find $z\in I$ such that $\psi=\phi(z)$.
Hence $a=\psi(g)=g(z)$ and $b=\psi(-f)=-f(z)$.
We obtain $\binom{a}{b}=\binom{g(z)}{-f(z)}=\binom{-g}{f}(-z)$.
Now the proof of the lemma is completed.
\end{proof}

An ideal $I$ of $R$ is called {\em Gorenstein} if the quotient ring $R/I$ is Gorenstein.
Now we state and show the second main result in this section.

\begin{thm}\label{3}
Let $R$ be a Gorenstein normal local ring of dimension $d$.
Let $I$ be a non-principal Gorenstein ideal of $R$ with height $1$.
Assume $I$ is rigid (this is the case if $d\ge3$ and $I$ is locally free on the punctured spectrum of $R$).
Then the point $nI\in\Cl(R)$ with $n\in\Z$ is maximal Cohen--Macaulay if and only if $n=-1,0,1$.
\end{thm}

\begin{proof}
The assertion of the theorem is evident in the case where $I$ is a principal ideal of $R$.
So let us assume that $I$ is not a principal ideal.
In particular, we have $d\ge2$.
As $R/I$ is a Cohen--Macaulay local ring of dimension $d-1$, by the depth lemma the $R$-module $I$ is maximal Cohen--Macaulay.
Hence $I$ is a reflexive ideal, and belongs to $\Cl(R)$.
Note that $\Hom_R(I,I)\cong R$.
If $d\ge3$ and $I$ is locally free on the punctured spectrum of $R$, then $I$ is rigid by Lemma \ref{4}.
Since $R/I$ is a $(d-1)$-dimensional Gorenstein local ring, we have $\Ext_R^1(R/I,R)\cong R/I$.
Dualizing by $R$ the natural exact sequence $0\to I\to R\to R/I\to0$, we get an exact sequence $0\to R\to I^\ast\xrightarrow{f}R/I\to0$.
Taking the pullback diagram of the map $f$ and the natural surjection $R\to R/I$, we obtain an exact sequence $0\to I\xrightarrow{\alpha}R^{\oplus2}\xrightarrow{\beta}I^\ast\to0$.
Using Lemma \ref{60}, we can write $\alpha=\binom{-g}{f}$ and $\beta=(f',g')$ with $f,g\in I^\ast$.

In what follows, we denote by $(n)$ a reflexive $R$-module $M$ of rank $1$ with $M=nI$ in $\Cl(R)$.
Hence there is an exact sequence $0\to(1)\xrightarrow{\binom{-g}{f}}(0)^{\oplus2}\xrightarrow{(f',g')}(-1)\to0$.
Applying $\Hom_R((1),-)$ gives an exact sequence
\begin{equation}\label{17}
0\to(0)\xrightarrow{\binom{-g'}{f'}}(-1)^{\oplus2}\to(-2)\to0,
\end{equation}
since $\Ext_R^1((1),(1))=0$ by the assumption that $I$ is rigid.
These two sequences make exact squares in the lower left, which give the exact square in the lower right; see \cite[Definition 2.1 and Lemma 2.2]{trans}.
$$
\xymatrix{
(1)\ar[r]^f\ar[d]^g & (0)\ar[r]^{f'}\ar[d]^{g'} & (-1)\ar[d] \\
(0)\ar[r]^{f'} & (-1)\ar[r] & (-2)
}\qquad\qquad\qquad
\xymatrix{
(1)\ar[r]^{f'f}\ar[d]^g & (-1)\ar[d] \\
(0)\ar[r] & (-2)
}
$$
Hence there is an exact sequence $0\to(1)\xrightarrow{\binom{-g}{f'f}}(0)\oplus(-1)\to(-2)\to0$.
Applying $\Hom_R((1),-)$ again gives an exact sequence $0\to(0)\xrightarrow{\binom{-g'}{(f'f)'}}(-1)\oplus(-2)\to(-3)\to\Ext_R^1((1),(1))=0$.
Again we get exact squares
$$
\xymatrix{
(1)\ar[r]^f\ar[d]^g & (0)\ar[r]^{(f'f)'}\ar[d]^{g'} & (-2)\ar[d] \\
(0)\ar[r]^{f'} & (-1)\ar[r] & (-3)
}\qquad\qquad\qquad
\xymatrix{
(1)\ar[r]^{(f'f)'f}\ar[d]^g & (-2)\ar[d] \\
(0)\ar[r] & (-3)
}
$$
and obtain an exact sequence $0\to(1)\xrightarrow{\binom{-g}{(f'f)'f}}(0)\oplus(-2)\to(-3)\to0$.
Iterating this procedure, we observe that for all integers $n\ge1$ there exists an exact sequence
\begin{equation}\label{5}
0\to(1)\xrightarrow{\binom{-g}{f^{(n)}}}(0)\oplus(-n+1)\to(-n)\to0,
\end{equation}
where $f^{(n)}$ stands for the map inductively defined by $f^{(1)}=f$ and $f^{(i)}=(f^{(i-1)})'f$ for $i\ge2$.

Suppose that there exists an integer $m\notin\{-1,0,1\}$ such that $(m)\in\cm(R)$.
Then $(-m)=(m)^\ast\in\cm(R)$, and so we may assume $m>0$, whence $m\ge2$.
It follows from \eqref{5} that $(-n)\in\cm(R)$ implies $(-n+1)\in\cm(R)$.
We inductively observe that $(-2)\in\cm(R)$.
The exact sequence \eqref{17} then splits since $\Ext_R^1(\cm(R),R)=0$, which yields an isomorphism $(0)\oplus(-2)\cong(-1)^{\oplus2}$.
Taking $(-)^\ast$, we get $(0)\oplus(2)\cong(1)^{\oplus2}$.
Completing this isomorphism and using the Krull--Schmidt theorem, we see that $\widehat I\cong\widehat R$, which implies $I\cong R$ by \cite[Corollary 1.15]{LW}.
Thus $I$ is a principal ideal of $R$, contrary to our assumption.
We conclude that the only integers $n$ satisfying $(n)\in\cm(R)$ are $-1,0,1$.
\end{proof}

\begin{rem}
In addition to the assumption of Theorem \ref{3}, suppose that $\Cl(R)$ is free.
Then the maximal Cohen--Macaulay points on $\Z I\subseteq\Cl(R)_\R$ are $0,I,-I$ as the map $\Cl(R)\to\Cl(R)_\R$ is injective.
\end{rem}

\begin{ex}\label{19}
\begin{enumerate}[(1)]
\item
Let $k,R,\p$ be as in Example \ref{27}.
Then $\Cl(R)=\Z\p$.
Theorem \ref{3} shows that the maximal Cohen--Macaulay points in $\Cl(R)$ are $0,\p,-\p$.
By \cite[(9.9) and Theorem (12.10)]{Y}, when $k$ is algebraically closed and has characteristic $0$, all the indecomposable maximal Cohen--Macaulay $R$-modules have rank $1$, whence $R,\p,\p^\ast$ are the indecomposable maximal Cohen--Macaulay $R$-modules.
\item
Let $R=k[\![x,y,z,w]\!]/(x^2w-yz)$ with $k$ a field.
Put $\p=(x,y)$, $\q=(x,y,z)$ and $\m=(x,y,z,w)$.
Suppose that $k$ is perfect.
Then $\sing R=\v(\q)=\{\q,\m\}$ by the Jacobian criterion, and it is seen that $R$ does not have an isolated singularity but is normal.
We have $\p^2=(x^2,xy,y^2)=(x^2,y)\cap(x^2,xy,y^2,z)$, which implies $\p^{(2)}=(x^2,y)$, and hence $\p^{(2)}$ is Gorenstein.
As $\p^{(2)}R_\q\cong R_\q$, the ideal $\p^{(2)}$ is locally free on the punctured spectrum of $R$.
Applying Theorem \ref{3} to $\p^{(2)}$, we see that the maximal Cohen--Macaulay points on $2\Z\p$ in $\Cl(R)$ are $0,\pm2\p$.
\end{enumerate}
\end{ex}

\begin{rem}
Let $R,\p$ be as in Example \ref{19}(2).
Suppose that Theorem \ref{3} holds for the Gorenstein ideal $\p$ of height one.
Then the maximal Cohen--Macaulay points on $\Z\p$ in $\Cl(R)$ are $0,\pm\p$.
However, we know that $\Z\p\ni2\p=\p^{(2)}=(x^2,y)$ is maximal Cohen--Macaulay, which is a contradiction.

In fact, Theorem \ref{3} cannot be applied to the ideal $\p$, since $\p$ is not rigid.
This can be checked as follows.
The minimal free resolution of the $R$-module $\p$ is the exact sequence in the lower left, and $\Ext_R^1(\p,\p)$ is the first cohomology of the complex in the lower right.
$$
0\gets\p\xleftarrow{\left(\begin{smallmatrix}
x&y
\end{smallmatrix}\right)}R^{\oplus2}\xleftarrow{\left(\begin{smallmatrix}
y&xw\\
-x&-z
\end{smallmatrix}\right)}R^{\oplus2}\xleftarrow{\left(\begin{smallmatrix}
-z&-xw\\
x&y
\end{smallmatrix}\right)}R^{\oplus2}\gets\cdots,\ 
(0\to\p^{\oplus2}\xrightarrow{\left(\begin{smallmatrix}
y&-x\\
xw&-z
\end{smallmatrix}\right)}\p^{\oplus2}\xrightarrow{\left(\begin{smallmatrix}
-z&x\\
-xw&y
\end{smallmatrix}\right)}\p^{\oplus2}\to\cdots).
$$
It is easy to verify that the vector $\binom{y}{xw}\in\p^{\oplus2}$ is a cycle which is not a boundary.
Hence $\Ext_R^1(\p,\p)\ne0$.
This argument assures that the rigidity assumption in Theorem \ref{3} is indispensable.
\end{rem}

The following result gives some specific information about maximal Cohen--Macaulay points on lines in $\Cl(R)$ of the form $J+\Z I$.

\begin{prop}
Let $R$ be a Gorenstein normal local ring of dimension $d\ge3$.
Let $I,J$ be Gorenstein ideals of $R$ with height one.
Suppose that $I$ is locally free on the punctured spectrum of $R$.
\begin{enumerate}[\rm(1)]
\item
If the point $J+nI\in\Cl(R)$ is maximal Cohen--Macaulay for some integer $n\ge0$ (resp. $n\le0$), then so is $J+iI$ for all integers $0\le i\le n$ (resp. $n\le i\le0$).
\item
The points $J+2I,\,J-2I$ in $\Cl(R)$ cannot simultaneously be maximal Cohen--Macaulay, unless $I=0$ in $\Cl(R)$ (i.e., unless $I$ is principal).
\end{enumerate}
\end{prop}

\begin{proof}
We reuse some parts of the proof of Theorem \ref{3}.
By the argument at the beginning, we may assume $I$ is nonprincipal, and $I,J$ are maximal Cohen--Macaulay modules.
For each element $x\in\Cl(R)$, we denote by $(x)$ a reflexive $R$-module $M$ of rank one with $M=x$ in $\Cl(R)$, whence $(0)\cong R$.

(1) It follows from \eqref{5} that there is an exact sequence $0\to(I)\to(0)\oplus((-n+1)I)\to(-nI)\to0$ for each $n>0$.
Applying the functor $\Hom_R((I),-)$, we get an exact sequence $0\to(0)\to(-I)\oplus(-nI)\to(-(n+1)I)\to0$ for all integers $n>0$.
Applying $\Hom_R((\pm J),-)$ gives an exact sequence
$$
0\to(\mp J)\to(\mp J-I)\oplus(\mp J-nI)\to(\mp J-(n+1)I)\to\Ext_R^1((\pm J),(0))=0,
$$
where the equality follows from the fact that $J,J^\ast$ are maximal Cohen--Macaulay $R$-modules.
This short exact sequence shows that if the $R$-module $(\mp J-(n+1)I)$ is maximal Cohen--Macaulay, then so is $(\mp J-nI)$.
Hence, for each integer $n>0$, if the point $J\pm(n+1)I\in\Cl(R)$ is maximal Cohen--Macaulay, then so is $J\pm nI$.
The assertion follows from this.

(2) Suppose that both $J+2I$ and $J-2I$ are maximal Cohen--Macaulay points of $\Cl(R)$.
The argument in the first paragraph of the proof of Theorem \ref{3} provides an exact sequence $0\to(J)\to(0)^{\oplus2}\to(-J)\to0$.
Applying $\Hom_R((-2I),-)$, we get an exact sequence
\begin{equation}\label{23}
0\to(2I+J)\to(2I)^{\oplus2}\to(2I-J)\to\Ext_R^1((-2I),(J))=0,
\end{equation}
where the equality follows from Lemma \ref{4}, since $\Hom_R((-2I),(J))=(2I+J)$ is maximal Cohen--Macaulay and $(-2I)=((I^{\otimes2})^{\ast\ast\ast})$ is locally free on the punctured spectrum of $R$.
By \eqref{23} the point $2I\in\Cl(R)$ is maximal Cohen--Macaulay, and so is $-2I$.
An analogous argument as in the last part of the proof of Theorem \ref{3} yields a contradiction.
Thus the proof is completed.
\end{proof}

\begin{ac}
The author is grateful to the referees for valuable comments.
Among other things, the author is deeply indebted to one of the referees for not only reading the paper carefully, but also giving the author a lot of useful suggestions, which have considerably polished up the paper.
Indeed, those suggestions have highly improved some of our results including Theorem \ref{d}, and have substantially simplified the proofs.

Part of this work was done during the author's visit to the University of Kansas from March 2018 to September 2019, and the author thanks them for their hospitality.
Also, the author thanks Hailong Dao, Futoshi Hayasaka, Yuji Kamoi, Kazuhiko Kurano, Hiroki Matsui, Tsutomu Nakamura, Koji Nishida, Shunsuke Takagi, Kei-ichi Watanabe and Ken-ichi Yoshida for their valuable and helpful comments.
\end{ac}



\begin{thebibliography}{99}
\bibitem{A}
{\sc L. L. Avramov}, Infinite free resolutions, {\em Six lectures on commutative algebra}, 1--118, Mod. Birkh\"auser Class., {\em Birkh\"auser Verlag, Basel}, 2010.
\bibitem{Bou}
{\sc N. Bourbaki}, Commutative algebra, Chapters 1--7, Translated from the French, Reprint of the 1989 English translation, Elements of Mathematics (Berlin), {\em Springer--Verlag, Berlin}, 1998.
\bibitem{B}
{\sc M. Brodmann}, The asymptotic nature of the analytic spread, {\em Math. Proc. Cambridge Philos. Soc.} {\bf 86} (1979), no. 1, 35--39.
\bibitem{BS}
{\sc M. P. Brodmann; R. Y. Sharp}, Local cohomology: an algebraic introduction with geometric applications, Cambridge Studies in Advanced Mathematics, {\bf 60}, {\em Cambridge University Press, Cambridge}, 1998.
\bibitem{Br}
{\sc A. Br\o ndsted}, An introduction to convex polytopes, Graduate Texts in Mathematics, {\bf 90}, {\em Springer--Verlag, New York--Berlin}, 1983.
\bibitem{BG}
{\sc W. Bruns; J. Gubeladze}, Semigroup algebras and discrete geometry, {\em Geometry of toric varieties}, 43--127, S\'emin. Congr. {\bf 6}, {\em Soc. Math. France, Paris}, 2002.
\bibitem{BH}
{\sc W. Bruns; J. Herzog}, Cohen--Macaulay rings, revised edition, Cambridge Studies in Advanced Mathematics, {\bf 39}, {\it Cambridge University Press, Cambridge}, 1998.
\bibitem{BV}
{\sc W. Bruns; U. Vetter}, Determinantal rings, Lecture Notes in Mathematics {\bf 1327}, {\em Springer--Verlag, Berlin}, 1988.
\bibitem{BGS}
{\sc R.-O. Buchweitz; G.-M. Greuel; F.-O. Schreyer}, Cohen--Macaulay modules on hypersurface singularities, II, {\em Invent. Math.} {\bf 88} (1987), no. 1, 165--182.
\bibitem{BD}
{\sc I. Burban; Y. Drozd}, Maximal Cohen--Macaulay modules over surface singularities, {\em Trends in representation theory of algebras and related topics}, 101--166, EMS Ser. Congr. Rep., {\em Eur. Math. Soc., Z\"urich}, 2008.
\bibitem{CL}
{\sc F. Call; G. Lyubeznik}, A simple proof of Grothendieck's theorem on the parafactoriality of local rings, {\em Commutative algebra: syzygies, multiplicities, and birational algebra (South Hadley, MA, 1992)}, 15--18, Contemp. Math., 159, {\em Amer. Math. Soc., Providence, RI}, 1994.
\bibitem{CK}
{\sc C-Y. J. Chan; K. Kurano}; The cone spanned by maximal Cohen--Macaulay modules and an application, {\em Trans. Amer. Math. Soc.} {\bf 368} (2016), no. 2, 939--964.
\bibitem{C}
{\sc L. W. Christensen}, Gorenstein dimensions, Lecture Notes in Mathematics, {\bf 1747}, {\em Springer--Verlag, Berlin}, 2000.
\bibitem{ncr}
{\sc H. Dao; O. Iyama; R. Takahashi; C. Vial}, Non-commutative resolutions and Grothendieck groups, {\em J. Noncommut. Geom.} {\bf 9} (2015), no. 1, 21--34.
\bibitem{qgor}
{\sc H. Dao; O. Iyama; R. Takahashi; M. Wemyss}, Gorenstein modifications and $\Q$-Gorenstein rings, {\em J. Algebraic Geom.} (to appear), {\tt arXiv:1611.04137v1}.
\bibitem{DK}
{\sc H. Dao; K. Kurano}, Boundary and shape of Cohen--Macaulay cone, {\em Math. Ann.} {\bf 364} (2016), no. 3-4, 713--736.
\bibitem{DS}
{\sc H. Dao; I. Shipman}, Representation schemes and rigid maximal Cohen--Macaulay modules, {\em Selecta Math. (N.S.)} {\bf 23} (2017), no. 1, 1--14.
\bibitem{DHM}
{\sc S. P. Dutta; M. Hochster; J. E. McLaughlin}, Modules of finite projective dimension with negative intersection multiplicities, {\em Invent. Math.} {\bf 79} (1985), no. 2, 253--291.
\bibitem{Fl}
{\sc H. Flenner}, Rationale quasihomogene Singularit\"aten, {\em Arch. Math. (Basel)} {\bf 36} (1981), no. 1, 35--44.
\bibitem{F}
{\sc R. M. Fossum}, The divisor class group of a Krull domain, Ergebnisse der Mathematik und ihrer Grenzgebiete, Band {\bf 74}, {\em Springer--Verlag, New York--Heidelberg}, 1973.
\bibitem{Fx}
{\sc H.-B. Foxby}, The MacRae invariant, {\em Commutative algebra: Durham 1981 (Durham, 1981)}, pp. 121--128, London Math. Soc. Lecture Note Ser. {\bf 72}, {\em Cambridge Univ. Press, Cambridge}, 1982.
\bibitem{GS}
{\sc S. Goto; Y. Shimoda}, On the Gorensteinness of Rees and form rings of almost complete intersections, {\em Nagoya Math. J.} {\bf 92} (1983), 69--88.
\bibitem{GW}
{\sc S. Goto; K. Watanabe}, On graded rings I, {\em J. Math. Soc. Japan} {\bf 30} (1978), no. 2, 179--213.
\bibitem{HS}
{\sc C. Huneke; I. Swanson}, Integral closure of ideals, rings, and modules, London Mathematical Society Lecture Note Series {\bf 336}, {\em Cambridge University Press, Cambridge}, 2006.
\bibitem{I}
{\sc O. Iyama}, Higher-dimensional Auslander--Reiten theory on maximal orthogonal subcategories, {\em Adv. Math.} {\bf 210} (2007), no. 1, 22--50.
\bibitem{K2}
{\sc K. Kurano}, Numerical equivalence defined on Chow groups of Noetherian local rings, {\em Invent. Math.} {\bf 157} (2004), no. 3, 575--619.
\bibitem{K}
{\sc K. Kurano}, The singular Riemann--Roch theorem and Hilbert--Kunz functions, {\em J. Algebra} {\bf 304} (2006), no. 1, 487--499.
\bibitem{LW}
{\sc G. J. Leuschke; R. Wiegand}, Cohen--Macaulay representations, Mathematical Surveys and Monographs, {\bf 181}, {\em American Mathematical Society, Providence, RI}, 2012.
\bibitem{M}
{\sc H. Matsumura}, Commutative ring theory, Translated from the Japanese by M. Reid, Second edition, Cambridge Studies in Advanced Mathematics, {\bf 8}, {\em Cambridge University Press, Cambridge}, 1989.
\bibitem{RS}
{\sc P. C. Roberts; V. Srinivas}, Modules of finite length and finite projective dimension, {\em Invent. Math.} {\bf 151} (2003), no. 1, 1--27.
\bibitem{R}
{\sc J. J. Rotman}, An introduction to homological algebra, Second edition, Universitext, {\em Springer, New York}, 2009.
\bibitem{catgp}
{\sc R. Takahashi}, Remarks on modules approximated by G-projective modules, {\em J. Algebra} {\bf 301} (2006), no. 2, 748--780.
\bibitem{trans}
{\sc R. Takahashi}, On the transitivity of degeneration of modules, {\em Manuscripta Math.} {\bf 159} (2019), no. 3-4, 431--444.
\bibitem{W}
{\sc K. Watanabe}, Rational singularities with $k^\ast$-action, {\em Commutative algebra (Trento, 1981)}, pp. 339--351, Lecture Notes in Pure and Appl. Math. {\bf 84}, {\em Dekker, New York}, 1983.
\bibitem{Y}
{\sc Y. Yoshino}, Cohen--Macaulay modules over Cohen--Macaulay rings, London Mathematical Society Lecture Note Series, {\bf 146}, {\em Cambridge University Press, Cambridge}, 1990.
\end{thebibliography}
\end{document}